\documentclass[11pt,a4paper]{article}

\usepackage{a4wide}
\usepackage{graphicx,color}
\usepackage{latexsym}
\usepackage{epsfig}
\usepackage{amssymb}
\usepackage{amstext}
\usepackage{amsgen}
\usepackage{amsxtra}
\usepackage{amsgen}
\usepackage{amsthm}
\usepackage{geometry}

\newtheorem{thm}{Theorem}[section]
\newtheorem{prop}[thm]{Proposition}
\newtheorem{lemma}[thm]{Lemma}
\newtheorem{cor}[thm]{Corollary}

\theoremstyle{definition}
\newtheorem{example}[thm]{Example}
\newtheorem{definition}[thm]{Definition}

\newtheorem{assumption}[thm]{Assumption}

\newtheorem{rem}[thm]{Remark}

\numberwithin{equation}{section}

\newcommand{\N}{\mathbb{N}}

\newcommand{\R}{\mathbb{R}}

\newcommand{\T}{\mathbb{T}}

\newcommand{\norm}[1]{\left\Vert #1 \right\Vert}

\newcommand{\E}{{\mathbb E}}
\def\tilde{\widetilde}

\begin{document}
\title{Large Noise in Variational Regularization }
\author{Martin Burger\thanks{Institut f\"ur Numerische und Angewandte Mathematik, Westf\"alische Wilhelms-Universit\"at (WWU) M\"unster. Einsteinstr. 62, D 48149 M\"unster, Germany. e-mail: martin.burger@wwu.de } \and Tapio Helin\thanks{Department of Mathematics and Statistics, University of Helsinki (UH). P.O. Box 68 (Gustaf H\"{a}llstr\"{o}min katu 2b)
FI-00014 University of Helsinki, Finland. e-mail: tapio.helin@helsinki.fi } \and Hanne Kekkonen \thanks{Cantab Capital institute, Faculty of Mathematics, University of Cambridge (UC), Wilberforce Road, Cambridge CB3 0WA, United Kingdom. e-mail: hnk22@cam.ac.uk}}
\maketitle
\begin{abstract}    
In this paper we consider variational regularization methods for inverse problems with large noise that is in general unbounded in the image space of the forward operator. We introduce a Banach space setting that allows to define a reasonable notion of solutions for more general noise in a larger space provided one has sufficient mapping properties of the forward operators.

A key observation, which guides us through the subsequent analysis, is that such a general noise model can be understood with the same setting as approximate source conditions (while a standard model of bounded noise is related directly to classical source conditions). Based on this insight we obtain a quite general existence result for regularized variational problems and derive error estimates in terms of Bregman distances. The latter are specialized for the particularly important cases of one- and $p$-homogeneous regularization functionals.

As a natural further step we study stochastic noise models and in particular white noise, for which we derive error estimates in terms of the expectation of the Bregman distance. The finiteness of certain expectations leads to a novel class of abstract smoothness conditions on the forward operator, which can be easily interpreted in the Hilbert space case. We finally exemplify the approach and in particular the conditions for popular examples of regularization functionals given by squared norm, Besov norm and total variation, respectively.

\end{abstract}

\newpage
\tableofcontents

\section{Introduction}
Motivated by stochastic modelling of noise, in particular white noise, the treatment of inverse problems with large noise has received strong attention recently \cite{egger,eggermont,kekkonen14,kekkonen15,mathe}. In this case large noise means that the norm of the data perturbation introduced by the noise is not small or might be even unbounded in the image space of the forward operator. 
Recently several papers have tackled such problems in the setting of linear regularization methods (corresponding to quadratic variational regularization), but also in those approaches some points were restrictive. The work by Eggermont et al. \cite{eggermont} assumes noise potentially large in the image space of the forward operator, but still being an element of this space. This allows to gain some insight, but still excludes white noise, where the latter condition is satisfied with probability zero. Also some difficulties related to the appropriate formulation of the regularized problem with white noise are not appearing in this way. Another line of research restricts to inverse problems with special settings of function spaces, namely some Sobolev spaces \cite{kekkonen14,kekkonen15} or Hilbert scales \cite{MairRuy96,MatheHof08,MathePer03,MatheTau06}. In these works estimates are obtained in weaker norms however and the setting still partly shadows the general structure.

In this paper we directly tackle the issue of large noise variational regularization with convex regularization functionals in Banach spaces. We derive a rather general theory that can be adapted to special homogeneity properties of the regularization functional, in particular to quadratic (Tikhonov) and one-homogeneous regularizations as popularized via total variation methods \cite{rudin1992nonlinear,tvzoo} and sparsity (see e.g. \cite{bruckstein2009sparse,daub07,RamlauTes06}). We consider the linear ill-posed problem
\begin{equation}
	\label{eq:main_eq}
	Ku = f,
\end{equation}
for a continuous linear operator $K: X \rightarrow Y$, where
$X$ and $Y$ are separable Banach and Hilbert space, respectively. 
For our setting of the noise let $(Z, Y, Z^*)$ be a Gelfand triple such that $Z\subset Y$ is a dense subspace with Banach structure and the dual pairing of $Z$ and $Z^*$ is compatible with the inner product of $Y$, i.e., by identifying $Y=Y^*$ we have
\begin{equation*}
	\langle u,v\rangle_{Z\times Z^*} = \langle u,v\rangle_Y
\end{equation*}
whenever $u\in Z \subset Y$ and $v\in Y = Y^* \subset Z^*$. The key assumption we make is that $K:  X \to Z$ is continuous. It directly follows that $K^*$ has a continuous extension $K^* : Z^* \to X^*$.
The noisy data are given by 
\begin{align}\label{eq:main_noisy}
f^\delta= Ku^\dagger + \delta n,
\end{align}
where $n \in Z^*$ and $\delta>0$ models the noise level.
Notice carefully that $f^\delta \in Z^*$ can be unbounded in the norm of $Y$, which yields our setting of large noise. It is crucial that due to the continuous extension property $K^* n$ is bounded in $X^*$.

As usual in variational methods we obtain a regularized solution of \eqref{eq:main_noisy} by computing a minimizer $u_{\alpha}^\delta$ of a weighted sum of the square residual (in the norm of $Y$) and the regularization functional. However, since the (squared) norm of $f^\delta$ is not necessarily finite, it is more appropriate to consider an expansion of the square residual \cite{kekkonen14,kekkonen15} and compute $u_\alpha^\delta$ as a minimizer of
\begin{equation} \label{minimisation}
	J^\delta_\alpha(u) = \frac{1}2 \Vert Ku \Vert_Y^2 - \langle K u, f^\delta \rangle_{Z\times Z^*} + \alpha R(u)
\end{equation}
with a convex regularization functional $R:X \rightarrow \mathbb{R} \cup \{\infty\}$. 

Our main assumptions on $R$ in addition to convexity are
\begin{itemize}
\item[(R1)] the functional $R$ is lower semicontinuous in some topology $\tau$ on X,
\item[(R2)] the sub-level sets $M_\rho = \{ R\leq \rho\}$ are sequentially compact in the topology $\tau$ on $X$ and
\item[(R3)] the convex conjugate $R^\star$ is finite on a ball in $X^*$ centered at zero.
\end{itemize}
The first two are the standard conditions needed for existence proofs and as we shall see below together with (R3) they will also lead to a general existence result for minimizers of $J_\alpha^\delta$ in the case of positive $\alpha$. Note that we assume that $K: X \rightarrow Z$ is continuous in $\tau$ topology. A standard example is $R$ being a power of a norm in a Banach space possessing a predual space. In this case the Banach--Alaoglu theorem yields compactness in the weak-star topology, for which we have genuine lower semicontinuity of the norm. We mention that a major difference to the case of bounded noise is that there is no natural lower bound for $J_\alpha^\delta$ (the lower bound in the case of bounded noise is $-\frac{1}2 \Vert f^\delta \Vert_Y^2 + \alpha R(u_0)$, with $u_0$ being a minimizer of $R$), which is the only complication in the analysis below and needs a suitable approximation of the noise together with (R3). To make some results below more accessible we will further employ the symmetry condition
\begin{itemize}
\item[(R4)] $R(-u)=R(u)$ for all $u \in X$,
\end{itemize}
which is however not essential for the overall line of arguments. 

Our key observation is related to error estimates between $u_\alpha^\delta$ and a solution $u^\dagger$ minimizing $R$ among all possible solutions of $Ku=f$. The usual way to obtain such is starting from the optimality condition for a minimizer  
\begin{equation}\label{optimality}
	K^* ( K u^\delta_\alpha - f^\delta) + \alpha \mu_\alpha^\delta = 0, \qquad \mu_\alpha^\delta \in \partial R(u^\delta_\alpha),
\end{equation}
where 
$$\partial R(u) = \{\mu \in X^* \; | \; R(u) - R(v) \leq \langle \mu, u-v\rangle_{X^*\times X}\; \textrm{for all}\; v\in X\}$$ 
stands for the subdifferential.
Next, the form (\ref{eq:main_noisy}) of $f^\delta$ is inserted and multiples of a subgradient $\mu^\dagger \in \partial R(u^\dagger)$ are added on both sides to arrive at
\begin{align}\label{optimalityrewritten}
 K^*K(u^\delta_\alpha - u^\dagger) + \alpha (\mu_\alpha^\delta - \mu^\dagger) = \delta \eta - \alpha \mu^\dagger, 
 \end{align}
where $\eta=K^*n \in X^*$. The following step is to take a duality product with $u_\alpha^\delta - u^\dagger$ and hence derive error estimates in the Bregman distance \cite{bregman,burgerbregman}. In doing so one can strongly benefit if $\mu^\dagger$ satisfies a source condition, i.e., if $\mu^\dagger = K^* w^\dagger$ for some $w^\dagger \in Y$. Note that in the bounded noise model also $\eta$ satisfies such a condition, which becomes violated in our setting. Since $\eta$ and $\mu^\dagger$ appear in a similar fashion on the right-hand side we see that the unboundedness of the noise in $Y$ leads to a similar technical issue as the violation of the source condition for $\mu^\dagger$. However, the latter is reasonably well understood and has been tackled by the concept of distance functions and approximate source conditions \cite{hein2008convergence,hofmann2006approximate,hofmann2005convergence,schusterbuch}, which are related to the growth rate of $\Vert w^\dagger \Vert_Y$ as $K^*w^\dagger$ approximates $\mu^\dagger$. Due to the analogous role of $\mu^\dagger$ and $\eta$ it is natural to use the same paradigm for approximating the large noise and this is the basic foundation of the analysis in this paper.

Following this idea our key contribution is to 
derive Bregman distance based error estimates between $u^\delta_\alpha$ and $u^\dagger$ for a general convex $R$. Given deterministic noise model, one can derive explicit converge rate results given (a variation of) an approximate source condition on $\mu^\dagger$ and $\eta$. In this paper, 
we prove convergence rates for the special cases of $1$-homogeneous $R(u)=\norm{u}_X$ as well as the $p$-homogeneous $R(u) = \frac 1p \norm{u}^p_X$ for $1<p<\infty$.

For our main motivation, random noise, the approximate source condition needs to be reconsidered in a statistical framework. In this work our interest lies in the frequentist risk between estimator $U^\delta_\alpha = U^\delta_\alpha(\omega)$ and the true unknown $u^\dagger$.
In such paradigm we find that the expected decay rate of the approximate source condition of the noise term is sufficient to guarantee a convergence rate result. Here we study and derive the convergence rate of frequentist risk for three examples: quadratic Tikhonov regularization, Besov norm regularization and total variation regularization. As for the noise we assume the canonical Gaussian white noise model on the Gelfand triplet $(Z,Y,Z^*)$ that has the well-known property that $n$ is almost surely unbounded in $Y$.

Let us shortly discuss some earlier work. After introducing the idea in \cite{burgerosher}, Bregman distances have been frequently used as an error measure for studying convergence rates of regularized solutions in Banach spaces. Convergence rates for the Bregman distance were further developed in e.g. \cite{benning2011error,bur07,grasmair2011linear,hofmann2007convergence,kindermannconvex,lorenz2008convergence,resmerita05,resmerita06}.
Iterative regularization based on Bregman distances were analysed e.g. in \cite{bur07,osh-bur-gol-xu-yin}. The literature on regularization theory in Banach spaces is quite extensive, but throughout the paper we often refer to an excellent textbook on the topic \cite{schusterbuch}. For a recent discussion of Bregman distances we refer to \cite{burgerbregman}. 

The general approach in statistical literature for solving frequentist inverse problems with white noise is typically based  on  obtaining  a  singular  value  decomposition  (SVD)  of  the  forward operator $K$ and then constructing a procedure based on spectral regularization,  see,  e.g. \cite{Abramovich1998, Cavalier2008, Cohen2004, Donoho1995, Goldenshluger2000, Kerkyacharian2000, Knapik2011}. However, in general inverse problems settings the SVD can rarely be computed analytically. Hence our approach which does not require the identification of the SVD basis of $K$ can be applied for wider range of inverse problems. 

The remainder of the paper is organized as follows: in Section \ref{sec:general_estimates} we consider the theory for general convex functional $R$. Main results of this section include the proof of existence of $u^\delta_\alpha$ as well as a related a-priori estimate in Section 2.1. The general error estimates are given in Section 2.4. In Section \ref{sec:convergence_rates} we derive convergence rates for different homogeneous examples of $R$. Next, we turn our focus on random noise in Section \ref{sec:random_noise} and consider examples of regularization by a quadratic Tikhonov functional (Section 4.2), Besov norm  (Section 4.3) and total variation functional (Section 4.4). Finally, we give an outlook to applications of our work to Bayesian inference in Section \ref{sec:outlook}.

\section{General Estimates}
\label{sec:general_estimates}

In the following we discuss the general approach for variational regularization under the assumptions above. We start by establishing the existence of a minimizer of $J_\alpha^\delta$ for $\alpha > 0$, which also yields some a-priori bounds for the solution.

\subsection{Existence and a-priori Estimates}
 
For general noise the existence of $J^\delta_\alpha$ is not clear from standard arguments. While usual lower semicontinuity arguments remains unchanged, the key issue is compactness, which follows from an a-priori estimate on $R$ due to the compactness of sublevel sets. In deriving such an estimate we need to bypass the missing lower bound of $J^\delta_\alpha$.

\begin{prop}\label{minestimate}  
Let $R$ satisfy the assumptions (R1)-(R4), then the functional $J^\delta_\alpha$ has a minimizer. Moreover, any such minimizer $u^\delta_\alpha$ satisfies
\begin{equation}
	R(u^\delta_\alpha) \leq \frac{1+\gamma}{1-\gamma} R(u^\dagger) + \frac{\delta^2}{2\alpha(1-\gamma)}  \Vert w \Vert_Y^2 +\frac{2\gamma}{1-\gamma}R^\star\left(\frac{\delta}{\alpha \gamma} (K^*w-\eta)\right)
\label{aprioriest}
\end{equation}
for any $\gamma \in (0,1)$ and $w \in Y$. Above $\eta=K^*n$.
\end{prop}

\begin{proof}
Consider the sublevel set $M = \{u\in X~|~J^\delta_\alpha(u) \leq J^\delta_\alpha(u^\dagger)\}$. Clearly, $M$ is non-empty since $u^\dagger \in M$. Further, any $u\in M$ satisfies
\begin{eqnarray*}
	\label{eq:existence_1}
	\frac{1}2 \Vert K (u-u^\dagger) \Vert_Y^2 + \alpha R(u) & \leq & \delta \langle K(u - u^\dagger), n \rangle_{Z\times Z^*} + \alpha R(u^\dagger) \nonumber \\
	& = & \delta \langle u - u^\dagger, \eta \rangle_{X\times X^*} + \alpha R(u^\dagger) \\
	& = & \delta \langle u - u^\dagger, \eta - K^*w \rangle_{X\times X^*} + \delta \langle K(u-u^\dagger),w \rangle_Y + \alpha R(u^\dagger)\\
	& \leq & \alpha \gamma R(u) + 2\alpha \gamma R^\star\left(\frac{\delta}{\alpha \gamma} (K^*w-\eta)\right) + \frac{1}2 \Vert K(u-u^\dagger)\Vert_Y^2  \\&& + \frac{\delta^2}2 \Vert w \Vert_Y^2 + \alpha (1+\gamma) R(u^\dagger),
\end{eqnarray*}
where $0 < \gamma < 1$ and $w \in Y$ is arbitrary. The last inequality follows from using generalized Young's inequality.  For the definition of the convex conjugate $R^\star$ see Appendix \ref{sec:convexconjugates}. Due to assumptions (R2), (R3) and $Y$ being dense in $Z^*$ we can now choose $w \in Y$ such that for a constant $C > 0$
$$ R^\star\left(\frac{\delta}{\alpha \gamma} (K^*w-\eta)\right) \leq C, $$
and hence we obtain
$$
	R(u) \leq \frac{1+\gamma}{1-\gamma} R(u^\dagger) + \frac{\delta^2}{2\alpha(1-\gamma)}  \Vert w \Vert_Y^2 +\frac{2C\gamma}{1-\gamma}
$$	
which implies $M$ is compact due to assumption (R2).   

Now the existence follows by standard arguments.
Without loss of generality we can assume that $\{u_j\}_{j=1}^\infty \subset M$ is a minimizing sequence of $J^\delta_\alpha$.
Since $M$ is compact, there exists a converging subsequence $u_{j_k} \to \tilde u \in X$. Finally, the lower semicontinuity of $J^\delta_\alpha$ yields that $\tilde u$ is a minimizer.
Note that with existence of a minimizer $\tilde u$ we directly obtain the a-priori estimate \eqref{aprioriest}.
\end{proof}

\begin{rem}
We can prove a similar a-prior estimate for $R$ also without the symmetry assumption (R4). In that case we get for the minimizer $u_\alpha^\delta$
\begin{multline*}
R(u^\delta_\alpha) \leq 
\frac{1+\gamma}{1-\gamma} R(u^\dagger) + \frac{\delta^2}{2\alpha(1-\gamma)}  \Vert w \Vert_Y^2 \\ 
+\frac{\gamma}{1-\gamma} \left( R^\star\left(\frac{\delta}{\alpha \gamma} (\eta-K^*w)\right) + R^\star\left(\frac{\delta}{\alpha \gamma} (K^*w-\eta)\right)\right).
\end{multline*}
\end{rem}

\subsection{Basic Ingredients of Error Estimates}

In the following we discuss some basics needed for the derivation of error estimates and the use of the approximate source conditions. The starting point for error estimates is the optimality condition mentioned above. Since the first two terms are linear and quadratic it is straight-forward to verify that they are Frechet-differentiable in our setting. Then the subdifferential of the whole functional equals the sum of the Frechet derivative of the first part and the subdifferential of the regularization functional (cf. \cite{ekeland1976convex}), which immediately implies the following statement:

\begin{prop}
Under the assumptions above, a minimizer $u_\alpha^\delta$ of $J_\alpha^\delta$ satisfies the optimality condition (\ref{optimality}).
\end{prop}

As mentioned above, error estimates are based on rewriting (\ref{optimality}) and then taking a duality product with $u_\alpha^\delta - u^\dagger$. This naturally leads to estimates in the Bregman distance, whose definition we recall for completeness:

\begin{definition}[Bregman distance]
Let $R:X\to \R\cup\{\infty\}$ be a convex functional. Then for each $\mu_v \in \partial R(v) \subset X^*$ we define generalised Bregman distance between $u$ and $v$ as
\begin{align*}
D_R^{\mu_v}(u,v)=R(u)-R(v)-\langle \mu_v, u-v\rangle_{X^*\times X}.
\end{align*}
Moreover, for $\mu_u\in \partial R(u)$ we define symmetric Bregman distance between $u$ and $v$ as
\begin{align}\label{Bregman1}
D_R^{\mu_u,\mu_v}(u,v)=\langle \mu_u-\mu_v, u-v\rangle_{X^*\times X}.
\end{align}
\end{definition}

Let us now sketch the basic steps in the derivation of error estimates and the standard route in the case of bounded noise. Taking a duality product with (\ref{optimality}) and $u_\alpha^\delta - u^\dagger$ we get $$ \Vert K(u_\alpha^\delta - u^\dagger) \Vert_Y^2 + \alpha D_R^{\mu_\alpha^\delta,\mu^\dagger}(u_\alpha^\delta,u^\dagger) \leq \langle  \delta \eta - \alpha \mu^\dagger, 
u_\alpha^\delta-u^\dagger\rangle_{X^*\times X}.$$

The nice case leading directly to estimates is $\eta=K^* n$ with $n\in Y$ and the additional source condition $\mu^\dagger = K^* w^\dagger \in X^*$ for $w^\dagger \in Y$. Then the right-hand side becomes 
$$  \langle \delta \eta - \alpha \mu^\dagger, u_\alpha^\delta-u^\dagger\rangle_{X^*\times X} = \langle \delta n - \alpha w^\dagger, K (u_\alpha^\delta - u^\dagger) \rangle_{Y}, $$
and Young's inequality implies 
$$ \frac{1}2 \Vert K(u_\alpha^\delta - u^\dagger) \Vert_Y^2 + \alpha D_R^{\mu_\alpha^\delta,\mu^\dagger}(u_\alpha^\delta,u^\dagger) \leq \frac{1}2 \Vert \delta n - \alpha w^\dagger \Vert_Y^2. $$

The problem now becomes more difficult if $\eta$ or $\mu^\dagger$ are not in the range of $K^*$ (if the range is defined as $K^* Y$ and not $K^*$ on a larger space including the noise). Note that with the notation using $\eta$ instead of $K^*n$ it becomes apparent that technically $\eta$ not in the range of $K^*$ is equally difficult as $\mu^\dagger$ not in the range of $K^*$. The latter case is however reasonably well understood, at least in the case of strictly convex functionals $R$. This is discussed in detail in \cite{schusterbuch}. The idea is to use a so-called approximate source condition, quantifying how well $\mu^\dagger$ can be
approximated by elements in the range of $K^*$. Since $\mu^\dagger$ needs to be in the closure of the range, there exists a sequence $w_n$ with $K^* w_n \rightarrow \mu^\dagger$. On the other hand it is not in the range, hence $\Vert w_n \Vert$ necessarily diverges.
Thus, one can measure how well $\mu^\dagger$, respectively in our case $\delta\eta - \alpha \mu^\dagger$ can be approximated by elements $K^* w$ with a given upper bound on $\Vert w \Vert$. The best estimates are then obtained by balancing errors containing the approximation of $\delta\eta-\alpha \mu^\dagger$ and $\Vert w \Vert$. 

In the case of no strict source condition and unbounded noise we will approximate $\mu^\dagger$ and $\eta$ with separate elements $K^* w_1$ and $K^* w_2$ respectively. Then we can write
\begin{multline*}
  \langle \delta \eta - \alpha \mu^\dagger, u_\alpha^\delta-u^\dagger\rangle_{X^*\times X} = \\
\langle \delta (\eta -K^*w_2) - \alpha (\mu^\dagger-K^*w_1), u_\alpha^\delta-u^\dagger\rangle_{X^*\times X}  +
 \langle \delta w_2 - \alpha w_1, K (u_\alpha^\delta - u^\dagger) \rangle_{Y},
\end{multline*}
where $w_1,w_2\in Y$. The second term on the right hand side can now be estimated using Young's inequality as above, while for the first term it is natural to apply the generalized Young's inequality as in the proof of Proposition \ref{minestimate}. We shall estimate the terms multiplied by $\delta$ and $\alpha$ separately and overall study a problem of estimating a term of the form $ \langle \eta, u_\alpha^\delta-u^\dagger\rangle_{X^*\times X}$. For this sake we could separately estimate the duality products with $u_\alpha^\delta$ and $u^\dagger$ as in the proof of Proposition \ref{minestimate}. However, as we are interested mainly in functionals with some homogeneity properties and in particular (R4) we shall see that it is beneficial to use the following direct estimate
\begin{equation}
	 \left\langle \eta, u_\alpha^\delta-u^\dagger\right\rangle_{X^*\times X} = \zeta \left\langle \frac{\eta}\zeta, u_\alpha^\delta-u^\dagger\right\rangle_{X^*\times X} \leq \zeta R\left(u_\alpha^\delta-u^\dagger\right) + 
	\zeta R^\star\left(\frac{\eta}\zeta\right) ,
\end{equation}
which we shall employ further with appropriately chosen $\zeta > 0$. We observe that in proceeding as above we are left with two terms in dependence on $w_1$, namely
$\frac{\alpha^2}2 \Vert w_1 \Vert^2$ and $\alpha \zeta R^\star(\frac{K^* w_1 - \mu^\dagger }\zeta)$. Analogous reasoning holds for $w_2$, with $\alpha$ replaced by $\delta$. This motivates our approach to the approximate source conditions to be detailed in the following.

\subsection{A Variation on Approximate Source Condition}

The standard concept of approximate source condition is to consider the case $R(u) = \Vert u \Vert_X^r$ for some power $r > 1$ (cf. \cite{schusterbuch}). The key concept is the so-called distance function
\begin{equation} \label{distance}  
	d_\rho(\vartheta) := \inf_{w \in Y} \{ \Vert K^*w - \vartheta \Vert_{X^*}~|~ \Vert w \Vert_Y \leq \rho \}, 
\end{equation}
and its asymptotics as $\rho \rightarrow \infty$. Note that in the case of a fulfilled source condition $d_\rho(\vartheta) = 0$ for $\rho$ sufficiently large, while in the really approximate case $d_\rho(\vartheta)$ decays to zero at a finite rate. Hence, the speed of decay of $d_\rho(\vartheta)$ is a natural measure to quantify the approximateness of the source condition. Unfortunately the existing theory employing the approximate source conditions or the even more implicit variational inequalities only works for the special norm-type functionals above (cf. \cite{schusterbuch}) and in addition uses some moduli of strict convexity of the norms. This of course excludes the most interesting cases of one-homogeneous regularizations such as sparsity and total variation. Hence we propose to consider a more general formulation based on convex duality. 

As we have seen above it is crucial to approximate some elements $\vartheta \in X^*$ by $K^*w$ with $w \in Y$ in some kind of Fenchel dual problem defined by $K$ and $R$. More precisely, we are interested in minimal values of the functional
\begin{equation*}
	E_{\alpha,\zeta}(w;\vartheta) =  \zeta R^\star\left(\frac{K^*w-\vartheta}\zeta\right) +  \frac{\alpha}2 \Vert w \Vert_Y^2,
\end{equation*}
which we shall denote as
\begin{equation}\label{eq:sourcecondition}
	e_{\alpha,\zeta}(\vartheta) = \inf_{w \in Y} E_{\alpha,\zeta}(w;\vartheta). 
\end{equation}
In this paper approximated source conditions correspond to determining decay rates for \eqref{eq:sourcecondition}. 

\begin{rem}
Indeed it can be inferred from the Fenchel duality theorem (cf. \cite{ekeland1976convex}) that $E_{\alpha,\zeta}(w;\vartheta)$ is dual (as a functional of $w$) to
\begin{equation*}
	F_{\alpha,\zeta}(v;\vartheta) = \frac{1}{2\alpha}\Vert K v \Vert_Y^2 - \langle \vartheta, v \rangle_{X^*\times X} + \zeta R(v)
\end{equation*}
and it holds that
\begin{equation}\label{eq:source}
	e_{\alpha,\zeta}(\vartheta) = - \inf_{v \in X} F_{\alpha,\zeta}(v;\vartheta).
\end{equation}
Thus, the measure $e_{\alpha,\zeta}$ measures how fast a regularization method approximating $\vartheta$ (related to the noise or source element) diverges and is hence a natural quantity.
For $R^\star(\vartheta)$ being finite, this immediately implies a bound on  $e_{\alpha,\zeta}(\vartheta)$ via the generalized Young inequality
$$ \langle \vartheta, v \rangle_{X^*\times X} \leq \frac{1}{\zeta} R^\star(\vartheta) + {\zeta} R(v). $$
This results into 
\begin{equation*}
	e_{\alpha,\zeta}(\vartheta) \leq \frac{1}{\zeta} R^\star(\vartheta).
\end{equation*}
Obviously, this estimate is not optimal under most conditions since it does not involve the first term in $F_{\alpha,\zeta}$. As we shall see below the bound can be improved under certain conditions, depending also on the homogeneity properties of $R$.
\end{rem} 

In the case of a Hilbert space regularization, $R(u) = \frac{1}2 \Vert u\Vert_X^2$, we have
$$E_{\alpha,\zeta}(w;\vartheta) = \frac{1}{2\zeta} \Vert \vartheta-K^*w \Vert_X^2  +  \frac{\alpha}2 \Vert w \Vert_Y^2 = \alpha E_{1,\zeta \alpha}(w;\vartheta) $$
and the problem of computing the minimizer is a classical Tikhonov regularization problem.
In particular in this example but also in the more general case the minimization of $E_{\alpha,\zeta}$ is closely related to the minimization in the definition of distance functions, roughly it can be understood as some kind of Lagrange multiplier formulation of the constrained problem for computing $d_\rho$. Moreover, it can be related to classical source conditions
$ \vartheta = (K^* K)^\nu w^*$, which are routinely used in the linear theory (cf. \cite{engl1996regularization}). 
We will provide other examples of approximate source conditions and their implications for functionals with a certain degree of homogeneity in Section 3.

We finally mention that we can also rewrite the a-priori estimate from Proposition \ref{minestimate} in terms of the approximate source condition \eqref{eq:sourcecondition}
\begin{equation}\label{eq:apriori_e}
	R(u_\alpha^\delta) \leq  \frac{1+\zeta}{1-\zeta} R(u^\dagger) + \frac{2\delta}{\alpha(1-\zeta)} e_{\frac{\delta}{2},\frac{\alpha \zeta}{\delta}}(\eta)
\end{equation}
with any $\zeta\in(0,1)$. Moreover, the approximate source conditions match well the use of Bregman distances as an error measure. Indeed, using the definition \eqref{Bregman1} of Bregman distance and completely analogous techniques as in the following one can show a conditional well-posedness result in the (symmetric) Bregman distances for all elements $u_1$, $u_2$ and their subgradients satisfying an approximate source condition and a bound
$R(u_1-u_2) \leq \gamma$, i.e. 
\begin{equation}
	D_R^{\mu_1,\mu_2}(u_1,u_2) \leq \varphi( \Vert Ku_1 - K u_2 \Vert),
\end{equation}
with 
$$ \varphi(t) = \inf_{\alpha, \zeta} \left( \frac{t^2}{2\alpha} + \gamma \zeta + e_{\alpha,\zeta}(\mu_1) +  e_{\alpha,\zeta}(\mu_2) \right) .$$

\subsection{Error Estimates}

In order to obtain error estimates we start from the rewritten version of the optimality condition
\eqref {optimalityrewritten}
and take a duality product with $u_\alpha^\delta - u^\dagger$ in the same way as sketched above. Then the right-hand side is estimated as 
\begin{multline}
	\langle \delta \eta -  \alpha \mu^\dagger , u_\alpha^\delta - u^\dagger \rangle_{X^*\times X} \leq \\ (  \alpha \zeta_1 + \delta \zeta_2 ) R(u_\alpha^\delta - u^\dagger) + \frac{1}2 \Vert K(u_\alpha^\delta - u^\dagger) \Vert_Y^2 +  \alpha e_{\alpha,\zeta_1}(\mu^\dagger)+ \delta e_{\delta,\zeta_2}(\eta)
\end{multline}
This immediately leads to the following error estimates:

\begin{prop}  \label{preliminaryestimates}
Let $R$ satisfy (R1)-(R4). Then with the assumptions above we obtain for any positive real numbers $\zeta_1$, $\zeta_2$:
\begin{equation}
	\Vert K(u_\alpha^\delta - u^\dagger) \Vert_Y^2 + 2 \alpha D_R^{\mu_\alpha^\delta,\mu^\dagger}(u_\alpha^\delta,u^\dagger) \leq 2 (  \alpha \zeta_1 + \delta \zeta_2 ) R(u_\alpha^\delta - u^\dagger) + 2
\alpha e_{\alpha,\zeta_1}(\mu^\dagger)+ 2 \delta e_{\delta,\zeta_2}(\eta)
\end{equation}
and furthermore 
\begin{equation}
	\label{eq:bregman_dist_est_first}
	D_R^{\mu_\alpha^\delta,\mu^\dagger}(u_\alpha^\delta,u^\dagger) \leq (   \zeta_1 + \frac{\delta}\alpha \zeta_2 ) R(u_\alpha^\delta - u^\dagger) + 
 e_{\alpha,\zeta_1}(\mu^\dagger)+ \frac{\delta}\alpha e_{\delta,\zeta_2}(\eta).
\end{equation}
\end{prop}

In order to obtain meaningful estimates we need to further estimate $R(u_\alpha^\delta - u^\dagger) $, ideally in terms of the Bregman distance, which however strongly depends on the specific scaling properties of the underlying functional $R$. Inspired by $p$-convex functionals (cf. \cite{bonesky2008minimization}), we shall consider the following assumption: There exists $\theta \in [0,1]$  such that
\begin{equation} \label{bregmanscaling}
R(u-v) \leq C_\theta(u,v) \left(D_R^{\mu_u,\mu_v}(u,v)\right)^\theta
\end{equation}
for all $u,v \in X$, $\mu_u \in \partial R(u)$ and $ \mu_v \in \partial R(v)$. Above  the constant $C_\theta$ is bounded on sets where $R(u)$ and $R(v)$ are bounded. The canonical examples to be considered are square norms (leading to $\theta =1$) and one-homogeneous functionals (leading to $\theta =0$).

\begin{example}\label{ex:two-homogeneous}
Let $X$ be a Hilbert space, $L$ a bounded linear operator, and $R(u)=\frac{1}2 \Vert Lu \Vert_X^2$. In consequence, $D_R^{\mu_u,\mu_v}(u,v)= \Vert L(u - v) \Vert_X^2 = 2 R(u-v)$ and inequality \eqref{bregmanscaling} holds with $\theta = 1$ and $C_\theta(u,v) \equiv \frac{1}2$.
\end{example}

\begin{example}
\label{ex:one-homogeneous}
Let $R$ be one-homogeneous, symmetric around zero, and convex. We immediately obtain a triangle inequality
$$ R(u-v) \leq R(u) + R(v), $$
and hence \eqref{bregmanscaling} holds with $\theta = 0$ and $C_0(u,v) = R(u) + R(v)$. It is easy to see that for $R$ of the above form no estimate with $\theta > 0$ can hold. As an example consider $R:\mathbb{R}\rightarrow \mathbb{R}$, $R(u)=|u|$. If $u$ and $v$ differ, but have equal sign, we obtain $|u-v|\neq 0$, but $D_R^{p,q}(u,v) = 0$.
\end{example}

\subsection{Convergence theorems}

With assumption \eqref{bregmanscaling} we can further estimate the right-hand side in the above estimates as
\begin{eqnarray*}
\left(   \zeta_1 + \frac{\delta}\alpha \zeta_2 \right) R(u_\alpha^\delta - u^\dagger) &\leq& 
\left(   \zeta_1 + \frac{\delta}\alpha \zeta_2 \right) C_\theta(u_\alpha^\delta ,u^\dagger) D_R^{\mu_\alpha^\delta,\mu^\dagger}(u_\alpha^\delta,u^\dagger)^\theta \\
&\leq& \theta D_R^{\mu_\alpha^\delta,\mu^\dagger}(u_\alpha^\delta,u^\dagger) + (1-\theta) \left(   \zeta_1 + \frac{\delta}\alpha \zeta_2 \right)^{1/(1-\theta)} C_\theta(u_\alpha^\delta ,u^\dagger)^{1/(1-\theta)},
\end{eqnarray*}
 if $\theta < 1$. In the case $\theta =1$ the first estimate is the only relevant one. This leads to the following result

\begin{thm} \label{deterministicestimates}
Let $R$ satisfy the assumptions of Proposition \ref{preliminaryestimates} and \eqref{bregmanscaling}. Then for, $\theta < 1$ we obtain 
\begin{multline}
\label{eq:deterministicestimates}
	D_R^{\mu_\alpha^\delta,\mu^\dagger}(u_\alpha^\delta,u^\dagger) \leq \inf_{(\zeta_1,\zeta_2) \in \mathbb{R}_+^2} \left\{ \left(   \zeta_1 + \frac{\delta}\alpha \zeta_2 \right)^{1/(1-\theta)} C_\theta(u_\alpha^\delta ,u^\dagger)^{1/(1-\theta)}\right. \\ \left.+ \frac{1}{1-\theta}
 e_{\alpha,\zeta_1}(\mu^\dagger)+ \frac{\delta}{\alpha(1-\theta)} e_{\delta,\zeta_2}(\eta)\right\}.
\end{multline}
For $\theta =1$ the estimate
\begin{equation}
	D_R^{\mu_\alpha^\delta,\mu^\dagger}(u_\alpha^\delta,u^\dagger) \leq \inf_{\zeta_1,\zeta_2 \in \Sigma} \frac{ e_{\alpha,\zeta_1}(\mu^\dagger)+ \frac{\delta}{\alpha } e_{\delta,\zeta_2}(\eta)}{1-(   \zeta_1 + \frac{\delta}\alpha \zeta_2 ) C_1(u_\alpha^\delta ,u^\dagger)}
\end{equation}
holds with 
$$ \Sigma = \left\{~(\zeta_1,\zeta_2) \in \mathbb{R}_+^2~\bigg\vert~ \left(   \zeta_1 + \frac{\delta}\alpha \zeta_2 \right) C_1(u_\alpha^\delta ,u^\dagger) < 1~\right\}. $$
\end{thm}

We finally mention an alternative statement of Theorem \ref{deterministicestimates}, which also takes into account an estimate of the residual. In the subsequent parts of the paper we will not discuss estimates for the residual, but obviously those can be obtained in the same way using the following result:
\begin{thm} \label{deterministicestimates2}
Let $R$ satisfy the assumptions of Proposition \ref{preliminaryestimates} and \eqref{bregmanscaling}. Then for, $\theta < 1$ we obtain 
\begin{multline*}
\|K(u_\alpha^\delta-u^\dagger)\|^2_Y + (2\alpha-\theta) D_R^{\mu_\alpha^\delta,\mu^\dagger}(u_\alpha^\delta,u^\dagger)\\
\leq \inf_{(\zeta_1,\zeta_2) \in (\mathbb{R}^+)^2} \left\{
(1-\theta)(2(\alpha\zeta_1 + \delta\zeta_2) )^{1/(1-\theta)} C_\theta(u_\alpha^\delta ,u^\dagger)^{1/(1-\theta)} 
 + 2\alpha e_{\alpha,\zeta_1}(\mu^\dagger)
+ 2\delta e_{\delta,\zeta_2}(\eta)\right\}
\end{multline*}
\end{thm}
Note that the constant $C_\theta(u_\alpha^\delta,u^\dagger)$ above depends on $R(u_\alpha^\delta)$ and hence also on the corresponding a-priori estimate.

\section{Convergence Rates for Homogeneous Regularizations}
\label{sec:convergence_rates}

Let us shortly introduce some notation. Throughout the following sections we denote $f \lesssim g$ for two functions if there exists a universal constant $C>0$ such that $f \leq Cg$ as functions. Moreover, if functions $f$ and $g$ are equivalent we write $f\simeq g$. Notice that if a random variable $X$ has a probability distribution $\pi$, we write $X \sim \pi$.  

\subsection{Regularization by one-homogeneous functionals}    
  
Let us directly proceed to the case of a one-homogeneous functional $R$ such as Besov-one norms or total variation.   
We assume that $X$ is a suitable space such that $R$ has a trivial nullspace (note that the nullspace of a one-homogeneous convex functional is always a linear space, and if it is finite-dimensional this component can be eliminated via similar arguments as in the total variation case detailed in \cite{tvzoo}).

In this case we can define a dual "norm" $S$ on $X^*$ via
\begin{equation}
	S(q) = \sup_{R(u) \leq 1} \langle q, u \rangle_{X^* \times X}.    \label{eq:Sdefinition}
\end{equation}
Note that $S$ is again one-homogeneous.
The one-homogeneity of $R$ implies
\begin{equation}
	\langle  q, u \rangle_{X^* \times X} \leq R(u) ~S(q)
\end{equation}
for all $ u \in X$ and $q \in X^*.$
In the case of one-homogeneous $R$ we can relate $R^\star$ and $S$ as follows:

\begin{lemma}\label{lemma:p=1}
Let $R: X \rightarrow \R \cup \{\infty \} $ be convex, non-negative and one-homogeneous and let $S: X^* \rightarrow \R \cup \{\infty \} $ be defined by \eqref{eq:Sdefinition}. Then for any $c \in \R^+$, we have
\begin{equation}
	R^\star(cq) = \left\{ \begin{array}{ll} 0 & \text{if } S(q) \leq \frac{1}c \\  +\infty & \text{else.} \end{array}  \right. 
\end{equation}
\end{lemma}
Note that under the convexity condition and the homogeneity $R(cu)=|c|R(u)$, that is, the  regularisation functional $R$ is sublinear. Hence, the proof follows from general results on sublinear functionals in \cite[Section V]{hiriart2013convex}. 
Next we formulate an alternative approximate source condition for the unknown and the noise term in one-homogeneous case. 

\begin{assumption}\label{assumption:p=1}
We assume to have an approximate source condition of order $r_1\geq 0$ for the unknown, that is we require 
\begin{align}\label{eq:1source}
\inf_{w \in Y} \left\{ \norm{w}_Y^2 ~\bigg\vert~S(\mu^\dagger -K^*w) \leq \beta\right\}=C_1\beta^{-r_1}
\end{align}
when $\beta>0$ small enough. We also require similar condition of order $r_2\geq0$ for the noise term and assume
\begin{align} \label{eq:1sourcenoise}
\inf_{w \in Y} \left\{ \norm{w}_Y^2 ~\bigg\vert~S(\eta -K^*w) \leq \beta \right\}=C_2\beta^{-r_2}.
\end{align}
\end{assumption}
Notice carefully that in the case when we do not have strict source condition the corresponding parameter $r_j$ must be strictly positive. Before proceeding, let us record the following technical lemma:

\begin{lemma}\label{lem:aux}    
The minimum of a problem $$M=\inf_{\zeta\in\R_+} (a \zeta^s + b \zeta^{-t})$$ for $a,b,s,t>0$ is achieved at 
\begin{equation}
	\label{lem:aux_min_value}
	\zeta = \left(\frac {bt}{as}\right)^{\frac 1{s+t}}
\end{equation}
yielding a minimum 
\begin{equation*}
	M \simeq a^{\frac t{s+t}} b^{\frac s{s+t}}.
\end{equation*}
\end{lemma}

\begin{proof}
Variational calculus yields 
\begin{equation*}
	a s \zeta^{s-1} - b t \zeta^{-t-1} = 0
\end{equation*}
at the minimum and hence \eqref{lem:aux_min_value} holds. Moreover, we obtain
\begin{equation*}
	M = a\left(\frac{bt}{as}\right)^{\frac s{s+t}} \left(1 + \frac st\right)
	\simeq a^{\frac t{s+t}} b^{\frac s{s+t}}.
\end{equation*}
\end{proof}

\begin{thm} \label{onehomthm1}
Let $X$ be a Banach space and  $R(u)=\|u\|_X$. Suppose that Assumption \ref{assumption:p=1} is satisfied with some orders $r_1,r_2\geq 0$. For the choice $\alpha\simeq\delta^\kappa$ where 
\begin{equation*}
\kappa = 
\begin{cases}
\frac{(1+r_1)(2+r_2)}{(2+r_1)(1+r_2)} & {\rm for}\; r_1 \leq r_2\; {\rm and}\\
1 & {\rm for}\; r_2<r_1,
\end{cases}
\end{equation*}
we have that
\begin{equation*}
	D_R^{\mu_\alpha^\delta,\mu^\dagger}(u_\alpha^\delta,u^\dagger) 
\lesssim
\begin{cases}
\delta^{\frac{2+r_2}{(2+r_1)(1+r_2)}} & {\rm for}\; r_1 \leq r_2\; {\rm and}\\
\delta^\frac{1}{1+r_1} & {\rm for}\; r_2<r_1.
\end{cases}
\end{equation*}
\end{thm}

\begin{proof}
Using Lemma \ref{lemma:p=1} we can write 
\begin{equation}
	\label{eq:one_homog_edelta}
	e_{\delta,\zeta}(\eta) = \frac \delta 2 \inf_{w \in Y} \left\{ \norm{w}_Y^2 ~\bigg\vert~S(\eta -K^*w) \leq \zeta \right\}.
\end{equation}
Recall from Example \ref{ex:one-homogeneous} that the one-homogeneous case corresponds to parameter $\theta=0$ in condition \eqref{bregmanscaling} and $C_0(u,v) = R(u)+R(v)$. The a priori estimate in Proposition \ref{aprioriest} gives us
\begin{align*}
R(u_\alpha^\delta)\leq \frac{1+\gamma}{1-\gamma}R(u^\dagger)+\frac{\delta^2}{2\alpha(1-\gamma)}\inf_{w\in Y}\{\|w\|_Y^2\ |\ S(K^*w-\eta)\leq\frac{\alpha\gamma}{\delta}\}
\end{align*} 
for any $\gamma\in[0,1)$ and $w\in Y$.
Now it follows from Theorem \ref{deterministicestimates} that
\begin{eqnarray}
	\label{eq:one_homog_D_R_est}
	D_R^{\mu_\alpha^\delta,\mu^\dagger}(u^\delta_\alpha,u^\dagger)
	& \leq & \inf_{\zeta_1,\zeta_2 \in \R_+^2}\left(\zeta_1 C_0(u^\delta_\alpha,u^\dagger)
	+ e_{\alpha,\zeta_1}(\mu^\dagger) +
	\frac \delta \alpha \zeta_2 C_0(u^\delta_\alpha,u^\dagger)
	+\frac{\delta}{\alpha} e_{\delta,\zeta_2}(\eta)\right) \nonumber \\
	&\lesssim &  M_1+M_2.
\end{eqnarray}
where 
\begin{align*}
M_1 = \inf_{\zeta_1\in \R_+}\bigg\{ \zeta_1\bigg(1+\frac{\delta}{\alpha} e_{\delta,\frac{\alpha\gamma}{\delta}}(\eta) \bigg)+e_{\alpha,\zeta_1}(\mu^\dagger)\bigg\}
\end{align*}
and 
\begin{align*}
M_2 = \inf_{\zeta_2\in \R_+}\bigg\{ \frac{\zeta_2\delta}{\alpha}\bigg(1+\frac{\delta}{\alpha} e_{\delta,\frac{\alpha\gamma}{\delta}}(\eta) \bigg)+\frac{\delta}{\alpha} e_{\delta,\zeta_2}(\eta)\bigg\}.
\end{align*}
From assumption \ref{assumption:p=1} we get the following estimates:
\begin{align*}
e_{\delta,\frac{\alpha\gamma}{\delta}}(\eta) & \lesssim \delta^{1+r_2}\alpha^{-r_2}\gamma^{-r_2},\\
e_{\alpha,\zeta_1}(\mu^\dagger) & \lesssim \alpha\zeta_1^{-r_1}, \\
e_{\delta,\zeta_2}(\eta) & \lesssim \delta\zeta_2^{-r_2}. 
\end{align*}
By assuming that $\alpha\simeq\delta^\kappa$ with some $\kappa>0$ we can write 
\begin{align*}
D_R^{\mu_\alpha^\delta,\mu^\dagger}(u^\delta_\alpha,u^\dagger) \leq & \inf_{\zeta_1 \in \R_+} \bigg\{(1+\delta^{r_3})\zeta_1+\delta^\kappa\zeta_1^{-r_1}\bigg\}+\inf_{\zeta_2 \in \R_+} \bigg\{\delta^{1-\kappa}(1+\delta^{r_3})\zeta_2+\delta^{2-\kappa}\zeta_2^{-r_2}\bigg\}.
\end{align*}
Above $r_3=2-\kappa+r_2(1-\kappa)>0$ when $\kappa\leq1$. 
Now by Lemma \ref{lem:aux} we get estimate 
\begin{align*}
M_1+M_2\simeq \delta^{\frac{\kappa}{1+r_1}}+\delta^{\frac{(1-\kappa)r_2+2-\kappa}{1+r_2}}.
\end{align*}
Optimizing the above we get $\kappa=\frac{(1+r_1)(2+r_2)}{(2+r_1)(1+r_2)}$ when $r_1\leq r_2$ which gives us the convergence rate 
\begin{align*}
D_R^{\mu_\alpha^\delta,\mu^\dagger}(u^\delta_\alpha,u^\dagger) 
 \lesssim \delta^{\frac{2+r_2}{(1+r_2)(2+r_1)}}.
\end{align*}
In the case $r_1\geq r_2$ we choose $\kappa=1$ to get 
\begin{align*}
D_R^{\mu_\alpha^\delta,\mu^\dagger}(u^\delta_\alpha,u^\dagger) 
 \lesssim \delta^{\frac{1}{1+r_1}}.
\end{align*}
\end{proof}

\begin{cor}
If in addition to the assumptions of Theorem \ref{onehomthm1} we assume the exact source condition for the unknown $u^\dagger$, i.e., $r_1=0$, we get a convergence rate
\begin{align*}
D_R^{\mu_\alpha^\delta,\mu^\dagger}(u^\delta_\alpha,u^\dagger) \simeq \delta^\frac{2+r_2}{2(1+r_2)}.
\end{align*}
\end{cor}
  
We finally provide some examples of one-homogeneous functionals and the meaning of the approximate source conditions in such cases:
\begin{example} \label{firstexample}
We start with a slightly artificial example, which however provides consistency with the linear theory. Assume $X$ is a Hilbert space and let $R(u) = \Vert u  \Vert_X$. Then $S(v) = \Vert v \Vert_X$ and for $u\neq 0$ we have $\partial R(u)=\left\{\frac{u}{\Vert u \Vert_X} \right\}$. In  \eqref{eq:1source} we thus look for the norm of $w$ when 
$$ \left\| \frac{u^\dagger}{\Vert u^\dagger \Vert_X} - K^* w \right\|_X \leq \beta. $$
Setting $\tilde w = w \Vert u^\dagger \Vert_X$, $\tilde \beta = \beta \Vert u^\dagger \Vert_X$, and $\tilde C_1 = C_1 \Vert u^\dagger \Vert^{r_1}_X$ we can reformulate the approximate source condition in terms of $\tilde \beta$ tending to zero as
$$
\inf_{\tilde w \in Y} \left\{ \norm{\tilde w}_Y^2 ~\bigg\vert~\Vert u^\dagger -K^*\tilde w\Vert_X \leq \tilde \beta\right\}=\tilde C_1 \tilde \beta^{-r_1},
$$ 
which is related to the approximate source condition for the regularization with the quadratic norm $\frac1{2}\Vert u \Vert_X^2$, whose subdifferential is $\{u\}$, see Assumption \ref{ass:p>1_source}.

We can further relate the approximate source condition to standard source conditions in the linear case. Assume that $K$ is a compact operator and let $u^\dagger = (K^* K)^\nu v$ for $v \in X$ and $0< \nu < \frac{1}2$. Then a simple calculation based on the singular value expansion shows that for each $\beta > 0$ there exists $w$ with $\Vert u^\dagger - K^* w \Vert_X \leq \beta$ and 
$\Vert w \Vert_Y \sim \beta^{1-1/2\nu}.$ Hence, an approximate source condition is satisfied with $r_1 \rightarrow 0$ as $\nu \rightarrow \frac{1}2$ and $r_1\rightarrow \infty$ as $\nu \rightarrow 0$.
\end{example}	

\begin{example} 
We proceed to one of the most canonical examples of a one-homogeneous functional, namely $X=\ell^1(\N)$, i.e. $u=(u_i)_{i=1}^\infty$, $u_i \in \R$, and 
$$ R(u) = \Vert u \Vert_X = \sum_{i=1}^\infty |u_i|. $$
In order to obtain a first insight we also consider a simple diagonal operator $K: \ell^1(\N) \rightarrow \ell^2(\N)$, $(u_i) \mapsto (k_i u_i)$, with a decreasing sequence $k_i$ of nonzero real values converging to zero. In addition we require that $K: \ell^2(\N) \to \ell^2(\N)$ is bounded.
 Then the dual norm is given by $S(v) = \Vert v \Vert_\infty$, i.e., in the definition of the source condition \eqref{eq:1source} for $\beta$ arbitrarily small we need to find $w$ such that
$$ \sup_{i \in \N} \vert \mu_i^\dagger - k_i w_i \vert = \Vert \mu^\dagger - K^*w \Vert_\infty = S(\mu^\dagger - K^*w) \leq \beta. $$
Now assume that $u^\dagger$ has an infinite support, i.e., there exists a nontrivial sequence $i_j$ with $u_{i_j}^\dagger \neq 0$ and hence $\vert \mu_{i_j}^\dagger \vert = 1$. Then we conclude for $i_j$ sufficiently large (note that $k_{i_j} w_{i_j}$ converges to zero)
$$ 1- |k_{i_j}|~|w_{i_j}| \leq \vert \mu_{i_j}^\dagger - k_{i_j} w_{i_j} \vert \leq \beta. $$
This implies $|w_{i_j}|  \geq \frac{1-\beta}{|k_{i_j}|}$ for $i_j$ sufficiently large, hence the corresponding sequence $w$ cannot be an element of $\ell^2(\N)$. Thus, the source condition can only be satisfied for $u^\dagger$ having a finite support. 

On the other hand, if $u^\dagger$ has finite support contained in $\{1,\ldots,M\}$ we can choose a subgradient $\mu^\dagger$ with $\mu_i^\dagger = 0$ for $i > M$. Then for element $w$ with $w_i = \frac{\mu^\dagger_i}{k_i}$ for $i \leq M$ and $w_i=0$ else, we have
$\mu^\dagger = K^*w$ and $w$ has finite $\ell^2$-norm, i.e., a standard source condition is satisfied. We thus see that in this case the asymptotic source condition does not seem useful, it is as strong as the original source condition due to the special structure of the subgradients. This behaviour is related to the degenerate behaviour of the $\ell^1$-regularization, which has some phase transition from a well-posed finite dimensional to an ill-posed infinite dimensional problem depending on the support (cf. \cite{grasmair2011linear,flemming2015l1}). We mention however that it is easy to see that the set of subgradients $\mu^\dagger$ for which the approximate source condition holds is larger for $r_1 > 0$ than for the standard source condition $r_1 = 0$, indeed the set is strictly increasing with $r_1$. The implication of this fact for the error estimation is not clear at this moment however.
%

We finally mention that approximate source conditions are useful in any case to quantify large noise as in \eqref{eq:1sourcenoise}, since the elements $\eta$ are then arbitrary and not characterized by the structure of subgradients. The condition simply measures how well the noise can be approximated in the $\ell^\infty$-norm by elements $K^*w$. 
\end{example} 

\begin{example}
A synthesis of the last two examples is group sparsity in Hilbert spaces. For simplicity let $H$ be a single Hilbert space and $X=\ell^1(\N;H)$, $u=(u_i)_{i=1}^\infty$, $u_i \in H$, with 
$$ R(u) = \sum_{i=1}^\infty \Vert u_i \Vert_H. $$
A subgradient $\mu \in \partial R(u)$ is given by $\mu=(\mu_1,\mu_2,\ldots)$ with $\mu_i \in H$ such that $\Vert \mu_i \Vert_H \leq 1$ and $\mu_i = \frac{u_i}{\Vert u_i \Vert_H}$ if $u_i \neq 0$. 
As in the previous example of $\ell^1$-regularization one can verify that an approximate source condition can only hold if only a finite number of the $u_i^\dagger$ are different from zero. On the other hand a source condition is not automatically satisfied in this case, we also need $\mu_i^\dagger = (K^* w)_i$, which requires analogous properties of the $u_i^\dagger$ as for $u^\dagger$ in Example \ref{firstexample}.
\end{example}

\begin{example} We finally provide a standard example as already used in \cite{burgerosher}, namely total variation denoising by the Rudin--Osher--Fatemi (ROF) functional (cf. \cite{rudin1992nonlinear}). This means we assume $D \subset \R^2$, $X=BV(D)$, $Y=L^2(D)$ and $K$ is the embedding operator between these spaces. The regularization functional is given by 
\begin{equation}
R(u) = \sup_{\varphi \in C_0^\infty(D), \Vert \varphi \Vert_\infty \leq 1} \int_{D}
\nabla \cdot \varphi u~dx.
\end{equation}
It is well-known that source conditions for ROF denoising are related to square integrability of the curvature of level sets (cf. \cite{burgerosher,chambolle2016geometric}). On the other hand the approximation properties of ROF are particularly bad if the exact solution is the characteristic function of a square, whose curvature is just a Radon measure on the jump set (cf. \cite{caselles2015total}). Hence, a natural conjecture is that approximate source conditions with $0<r_1<\infty$ are related to $q$-integrability of the curvature of level sets for $1<q<2$, which we make more explicit in the following. Assume for this sake that $u$ is the indicator function of a simply connected compact subset $D_0 \subset D$, such that $\Gamma = \partial D_0$ is of class $C^1$ and the curvature $\kappa$ is an element of $L^q(\Gamma)$. An elementary computation then yields that the normal and tangent fields are H\"older continuous along $\Gamma$ with exponent $\gamma=1-\frac{1}q$. Using this kind of regularity one obtains that the signed distance function $b_\Gamma$ is of class $C^{1,\gamma}$ in a neighbourhood of $\Gamma$ and the curvature of level sets $\Delta b_\Gamma $ is $q$-integrable in this neighbourhood. Now, similar to \cite{burgerosher,tvzoo}, we can construct a subgradient $\mu$ of the form 
$$\mu = \nabla \cdot g, \quad g = \psi(b_\Gamma) \nabla b_\Gamma  $$
with $\psi$ be a continuously differentiable function with local support around zero, $\psi(0)=1$ and $0 \leq \psi \leq 1$ else. 
For this subgradient we easily verify $\Vert g \Vert_{L^\infty} \leq 1$ and $\Vert \mu \Vert_{L^p} < \infty$.   

The dual norm of $BV$ given by
$$ S(v) = \inf\{\Vert h \Vert_{L^\infty}~|~ \nabla \cdot h = v \}, $$
thus in the  approximate source condition we know that $S(\mu-K^*w) \leq \beta$ as soon as we find any $h$ with $ \Vert h -g \Vert_{L^\infty} \leq \beta$. A sufficient condition for the approximate source condition is thus
$$ \inf \{ \Vert \nabla \cdot h \Vert_{L^2}~|~ \Vert h -g \Vert_{L^\infty} \leq \beta \} \leq C \beta^{-r_1}. $$
To verify such a condition let $G$ be a standard kernel with unit integral such as the Gaussian, $G_\epsilon = \epsilon^{-2} G(\frac{\cdot}\epsilon)$ and $h=G_\epsilon*g$. Then it is a standard computation for convolutions to show that 
$$ \Vert G_\epsilon*g - g \Vert_{L^\infty} \leq C_1 \epsilon^\gamma, \qquad \Vert \nabla \cdot G_\epsilon*g \Vert_{L^2} \leq C_2 \epsilon^{2(1-2/q)}=C_2 \epsilon^{2(2\gamma -1)}. $$
For $q >1$ we obtain $\gamma > 0$ and hence we can choose $\beta \sim \epsilon^\gamma$, which implies an approximate source condition with $r_1 = \frac{2}\gamma-4 = \frac{4-2q}{q-1}.$ With $q=2$ we recover the standard source condition for square integrable curvature, with $q \rightarrow 1$ we obtain $r_1 \rightarrow \infty$.
\end{example}

\subsection{Regularization by $p$-homogeneous functional for $1<p<\infty$}

In this section we consider regularization with functionals of type $R(u) = \frac 1 p \norm{u}^p_X$ for $1<p<\infty$.
Below $p,q\in(1,\infty)$ are H\"older conjugates, i.e., 
\begin{align*}
\frac{1}{p}+\frac{1}{q}	=1. 
\end{align*} 
Here we utilize additional assumptions regarding the Banach space $X$. Let $J_p: X \to X^*$ denote the set-valued duality mapping
\begin{equation*}
	J_p(u) = \{\mu \in  X^* \; | \; \langle \mu, u \rangle_{X^*\times X} = \norm{u}_X \norm{\mu}_{X^*} \quad {\rm and} \quad
\norm{\mu}_{X^*} = \norm{u}^{p-1}_X\}.
\end{equation*}
A Banach space $X$ is said to be \emph{$p$-convex} if there exists a constant $c_p>0$ such that
\begin{equation*}
	\frac 1 p \norm{u-v}_X^p \geq \frac 1 p \norm{u}_X^p - \left\langle j_p^X(u),v \right\rangle_{X^*\times X} +\frac{c_p}{p}\norm{v}^p_X
\end{equation*}
for all $u,v\in X$ and all $j_p \in J_p$.
Moreover, $X$ is called \emph{$p$-smooth} if there exists a constant $G_p>0$ such that 
\begin{equation*}
	\frac 1 p \norm{u-v}^p_X \leq \frac 1 p \norm{u}^p_X - \left\langle j_p^X(u),v\right\rangle_{X^*\times X} + \frac {G_p}p \norm{v}^p_X
\end{equation*}
for all $u,v\in X$ and all $j_p \in J_p$. The basic consequences and properties of these geometrical assumptions are listed in \cite{schusterbuch}. 
For what follows, an important connection between the convexity and smoothness assumptions is given in \cite[Thm 2.52]{schusterbuch}:
$X$ is p-smooth if and only if $X^*$ is $q$-convex. Moreover, $X$ is p-convex if and only if $X^*$ is $q$-smooth. Some examples of $\max \{2,p\}$-convex and $\min\{2,p\}$-smooth spaces are sequence spaces $\ell^p$, Lebesgue spaces $L^p$, and Sobolev spaces $W^{m,p}$. Notice also that in this Section we consider a $p$-smooth Banach space $X$ for some $p>1$. In that case it is well known (see \cite[Remark 2.38]{schusterbuch}) that the duality mapping $J_p$ is single-valued.

Next we define an alternative approximate source condition for the unknown and noise in case $R(u)=\frac{1}{p}\norm{u}_X^p$ for $1<p<\infty$.
\begin{assumption}\label{ass:p>1_source}
We assume to have an approximate source conditions of order $r_1\geq0$ for the unknown, i.e., we require that
\begin{align*}
\inf_{w\in Y}\Big\{\frac{1}{\beta}\|K^*w-\mu^\dagger\|^q_{X^*}+\frac{1}{2}\|w\|_Y^2\Big\}\leq C\beta^{-r_1}
\end{align*}
when $\beta>0$ is small enough. We also require a similar condition of order $r_2\geq0$ for the noise term and assume
\begin{align*}
\inf_{w\in Y}\Big\{\frac{1}{\beta}\|K^*w-\eta\|^q_{X^*}+\frac{1}{2}\|w\|_Y^2\Big\}\leq C\beta^{-r_2}.
\end{align*}
\end{assumption}
For comparison of the above approximate source condition to distance function see Remark \ref{Rem:comparison}.

\subsubsection*{Case $1<p<2$}  

\begin{thm} \label{generalpthm}
Suppose that the Banach space $X$ is $p$-smooth and $2$-convex and $R(u) = \frac 1p \norm{u}^p_X$ for some $1<p<2$. Moreover, suppose that Assumption \ref{ass:p>1_source} is satisfied with some orders $r_1, r_2\geq 0$ and $r_1<1$. Then for the choice  $\alpha\simeq\delta^\kappa$ where 
\begin{equation*}
\kappa =    
\begin{cases} 
\frac{\nu_1\nu_2}{\nu_1\nu_2+q(r_2-r_1)} & {\rm for}\; r_1 \leq r_2\; {\rm and}\\
1 & {\rm for}\; r_2<r_1<1
\end{cases}
\end{equation*}
we have convergence
\begin{equation*}
	D_R^{\mu_\alpha^\delta,\mu^\dagger}(u_\alpha^\delta,u^\dagger) 
\leq 
\begin{cases}
C_p\delta^{\frac{2\nu_2(1-r_1)}{\nu_1\nu_2+q(r_2-r_1)}} & {\rm for}\; r_1 \leq r_2\; {\rm and}\\
C_p\delta^\frac{2(1-r_1)}{2+r_1 (q-2)} & {\rm for}\; r_2<r_1<1.
\end{cases}
\end{equation*}
Above we have denoted $\nu_i=2+r_i(q-2)$ and $q=\frac{p}{p-1}$. For the constant $C_p$ we have $C_p\to\infty$ when $p\to 2$.
\end{thm}

\begin{proof}
We can apply the Xu--Roach inequality II \cite[Thm. 2.40 (b)]{schusterbuch} in $X$ to obtain
\begin{equation*}
	D^{\mu_u, \mu_v}_R (u,v) = \langle j_p(u)-j_p(v),u-v \rangle_{X^*\times X}
		\geq C \max\{\norm{u}_{X}, \norm{v}_{X}\}^{p-2} \norm{u - v}^2_{X}.
\end{equation*}
This gives us an estimate 
\begin{align*}
R(u-v)\leq C_{\frac{p}{2}}(u,v) D^{\mu_u, \mu_v}_R (u,v)^{\frac{p}{2}}
\end{align*}
with
\begin{equation*}
	C_{\frac{p}{2}}(u,v) = \frac{C}{p}\max\left\{\norm{u}_{X}, \norm{v}_{X}\right\}^{\frac{p(2-p)}{2}}.
\end{equation*}
By applying the trivial upper bound $\max\left\{\norm{u}_{X}, \norm{v}_{X}\right\} \leq \norm{u}_{X} + \norm{v}_{X}$ and the a priori bound given in (\ref{eq:apriori_e}) for any $\gamma\in(0,1)$ we have  
\begin{equation*}
	C_{\frac{p}{2}}(u^\delta_\alpha, u^\dagger) \leq C\Big( \|u^\dagger\|^p_X + \frac{\delta}{\alpha} e_{\frac{\delta}{2},\frac{\alpha\gamma}{\delta}}(\eta))\Big)^{\frac{2-p}{2}}.
\end{equation*}
Considering Theorem \ref{deterministicestimates} we now obtain
\begin{eqnarray}
\label{eq:1<p<2_first_breg_estimate}  
D_R^{\mu_\alpha^\delta,\mu^\dagger}(u_\alpha^\delta,u^\dagger) & \lesssim &
  \inf_{(\zeta_1,\zeta_2) \in \mathbb{R}_+^2} 
 \bigg\{ \bigg(\zeta_1+ \frac{\delta}\alpha \zeta_2 \bigg)^{\frac{2}{2-p}} C_{\frac{p}{2}}(u^\delta_\alpha, u^\dagger)^{\frac{2}{2-p}}
+ \frac{2}{2-p}\Big(e_{\alpha,\zeta_1}(\mu^\dagger)+ \frac{\delta}{\alpha} e_{\delta,\zeta_2}(\eta)\Big)\bigg\}\nonumber\\
& \lesssim  & \!\!\! M_1+ M_2.
\end{eqnarray}
Above we have for $s= \frac 2 {2-p}$ that
\begin{align*}
M_1= \inf_{\zeta_1 \in \R_+} \left(\zeta_1^s \bigg(1+\frac{\delta}{\alpha}e_{\frac{\delta}{2},\frac{\alpha\gamma}{\delta}}(\eta)\bigg)+ e_{\alpha,\zeta_1}(\mu^\dagger)\right)
\end{align*}
and
\begin{align*}
M_2= \inf_{\zeta_2 \in \R_+}\left(\zeta_2^s \left(\frac{\delta}\alpha\right)^{s} \bigg(1 + \frac{\delta}{\alpha}e_{\frac{\delta}{2},\frac{\alpha\gamma}{\delta}}(\eta)\bigg)+ \frac{\delta}{\alpha} e_{\delta,\zeta_2}(\eta)\right).
\end{align*}
Since $R(u)=\frac{1}{p}\|u\|_X^p$ we can write 
\begin{align*}
e_{\alpha,\zeta}(\eta)
& = \inf_{w\in Y}\Big\{\zeta R^\star\Big(\frac{1}{\zeta} (K^*w-\eta)\Big)+\frac{\alpha}{2}\|w\|_Y^2\Big\}\\
& = \alpha \inf_{w\in Y}\Big\{\frac{1}{q\alpha\zeta^{q-1}}\|K^*w-\eta\|_{X^*}^q+\frac{1}{2}\|w\|_Y^2\Big\}.
\end{align*}
From Assumption \ref{ass:p>1_source} we directly obtain following estimates:
\begin{eqnarray}\label{eq:p<2_estimates}
e_{\frac{\delta}{2},\frac{\alpha\gamma}{\delta}}(\eta) & \lesssim & \gamma^{-t_2} \delta^{1-r_2+t_2} \alpha^{-t_2}, \nonumber\\
e_{\alpha,\zeta_1}(\mu^\dagger) & \lesssim & \alpha^{1-r_1} \zeta_1^{-t_1}, \nonumber\\
e_{\delta,\zeta_2}(\eta) & \lesssim & \delta^{1-r_2} \zeta_2^{-t_2},
\end{eqnarray}
where we have set $t_i=(q-1)r_i\geq0$. By assuming further that $\alpha \simeq \delta^\kappa$ for some $\kappa>0$, we can reduce the two upper-most estimates to
\begin{equation*}\label{source_estimate}
	e_{\frac{\delta}{2},\frac{\alpha\gamma}{\delta}}(\eta) \lesssim \delta^{1-r_2+(1-\kappa)t_2} \quad {\rm and}
	\quad e_{\alpha,\zeta_1}(\mu^\dagger)  \lesssim \delta^{\kappa(1-r_1)}\zeta_1^{-t_1}.
\end{equation*}

Applying all the estimates above to Bregman distance in \eqref{eq:1<p<2_first_breg_estimate} we get
\begin{align*}
D_R^{\mu_\alpha^\delta,\mu^\dagger}(u_\alpha^\delta,u^\dagger)  \lesssim
& \inf_{\zeta_1 \in \R_+}
\left\{(1+\delta^{r_3})\zeta_1^s + 
\delta^{\kappa(1-r_1)}\zeta_1^{-t_1}\right\} + \\
&  \inf_{\zeta_2 \in \R_+} \left\{\delta^{(1-\kappa)s}(1+ \delta^{r_3})\zeta_2^s
 + \delta^{2-\kappa-r_2} \zeta_2^{-t_2} \right\}
\end{align*}
where we have assumed that $r_3 = 1-r_2+(1-\kappa)(t_2+1)\geq 0$. Further, applying Lemma \ref{lem:aux} yields us 
\begin{equation*}
	M_1 \simeq  \delta^{\kappa(1-r_1)\frac s {s+t_1}}
\end{equation*}
and
\begin{equation*}
	M_2 \simeq \delta^{(1-\kappa)\frac{st_2}{s+t_2}}\delta^{(2-\kappa-r_2)\frac{s}{s+t_2}}.
\end{equation*}
Consequently, we can reduce the estimate to
\begin{align*}
	M_1 + M_2 & \simeq \delta^{\frac{2\kappa(1-r_1)}{2+(q-2)r_1}}
	+ \delta^{\frac{2(2+r_2(q-2)-\kappa(r_2(q-1)+1))}{2+(q-2)r_2}}\\
	& =  \delta^{\frac{2\kappa(1-r_1)}{\nu_1}}
	+ \delta^{\frac{2(\nu_2-\kappa(r_2(q-1)+1))}{\nu_2}}
\end{align*}
where $\nu_i=2+r_i(q-2)$. 

When $r_1\leq r_2$ the above expression is minimized at 
\begin{equation*}
	\kappa = \frac{\nu_1\nu_2}{\nu_1\nu_2+q(r_2-r_1)}
\end{equation*}
yielding convergence rate 
\begin{equation*}
D_R^{\mu_\alpha^\delta,\mu^\dagger}(u_\alpha^\delta,u^\dagger) \leq C_p\delta^{\frac{2\nu_2(1-r_1)}{\nu_1\nu_2+q(r_2-r_1)}}.
\end{equation*}
In order to attain convergence we have to assume $r_1<1$.  
If $r_2 < r_1<1$ the optimal convergence rate is achieved when $\kappa=1$ and we obtain
\begin{equation*}
D_R^{\mu_\alpha^\delta,\mu^\dagger}(u_\alpha^\delta,u^\dagger) 
\leq C_p\delta^{\min\left\{\frac{2(1-r_1)}{2+r_1 (q-2)},\frac{2(1-r_2)}{2+r_2 (q-2)}\right\}}
= C_p\delta^\frac{2(1-r_1)}{2+r_1 (q-2)}.
\end{equation*}
Above the constant $C_p\to\infty$ when $p\to2$. Note that with the chosen $\kappa$ the assumption $r_3\geq0$ is always true when $r_1<1$.
\end{proof}

\begin{cor}
If in addition to the assumptions of Theorem \ref{generalpthm} the exact source condition $r_1=0$ is satisfied the estimate
\begin{align*}
D_R^{\mu_\alpha^\delta,\mu^\dagger}(u_\alpha^\delta,u^\dagger) & \leq C_p\delta^\frac{2\nu_2}{2\nu_2+qr_2}
\end{align*} 
holds. Furthermore, notice that assuming an exact source condition on the noise leads to the standard convergence rate of ${\mathcal O}(\delta)$ in the classical setting \cite{schusterbuch}.
\end{cor}

\begin{rem}
Let us illustrate another bound for $R(u-v)$ obtained via the Xu--Roach inequalities.
Since $X^*$ is $q$-convex and $2$-smooth \cite[Thm 2.52 (b)]{schusterbuch} we can apply \cite[Lemma 2.63]{schusterbuch} and the Xu--Roach inequality II \cite[Thm. 2.40 (b)]{schusterbuch} in $X^*$ to obtain
\begin{equation*}
	D^{\mu_u, \mu_v}_R (u,v) = D^{u,v}_{R^\star} (\mu_u, \mu_v)
	\geq C \max\{\norm{\mu_u}_{X^*}, \norm{\mu_v}_{X^*}\}^{q-q} \norm{\mu_u - \mu_v}^q_{X^*}= C\norm{\mu_u - \mu_v}^q_{X^*}.
\end{equation*}
Next by Xu--Roach inequality IV \cite[Thm. 2.42]{schusterbuch} we obtain
\begin{equation*}
	\norm{\mu_u-\mu_v}_{X^*} \geq C \max\{\norm{\mu_u}_{X^*}, \norm{\mu_v}_{X^*}\}^{2-q} \norm{u-v}_X,
\end{equation*}
where we have considered the inequality in $X^*$ which is $2$-smooth by assumption. 

Combining the two inequalities above yields
\begin{equation*}
	R(u-v)  \leq \frac{C}{p}\max\{\norm{u}_{X}, \norm{v}_{X}\}^{p(2-p)} D^{\mu_u, \mu_v}_R(u,v)^{\frac p q}
\end{equation*}
since $p-q+(2-p)q = 0$.
\end{rem}

\subsubsection*{Case $p=2$}

Finally we simplify the estimates in the quadratic case:

\begin{thm}
Suppose that $X$ is a Banach space and $R(u) = \frac 12 \norm{u}^2_X$. Moreover, suppose that Assumption \ref{ass:p>1_source} is satisfied with some orders $r_1, r_2\geq 0$ and $r_1<1$. For the choice $\alpha\simeq\delta^\kappa$, where 
\begin{equation*}
\kappa = 
\begin{cases}
\frac{2}{2+r_2-r_1} & {\rm for}\; r_1 \leq r_2\; {\rm and}\\
1 & {\rm for}\; r_2<r_1<1,
\end{cases}
\end{equation*}
we get convergence
\begin{equation*}
	D_R^{\mu_\alpha^\delta,\mu^\dagger}(u_\alpha^\delta,u^\dagger) 
\lesssim 
\begin{cases}
\delta^{\frac{2(1-r_1)}{2+r_2-r_1}} & {\rm for}\; r_1 \leq r_2\; {\rm and}\\
\delta^{1-r_1} & {\rm for}\; r_2<r_1<1.
\end{cases}
\end{equation*}
\end{thm}

\begin{proof}
Recall from the Example \ref{ex:two-homogeneous} that case $R(u)=\frac{1}{2}\|u\|_X^2$ corresponds to parameter $\theta=1$ and $C_2(u_\alpha^\delta,u^\dagger)=\frac{1}{2}$ in condition (\ref{bregmanscaling}). Hence the second part of the Theorem \ref{deterministicestimates} gives us 
\begin{align*}
D_R^{\mu_\alpha^\delta,\mu^\dagger}(u_\alpha^\delta,u^\dagger) & \lesssim 
\inf_{\zeta_1,\zeta_2 \in \Sigma}  \frac{e_{\alpha,\zeta_1}(\mu^\dagger)+ \frac{\delta}{\alpha } e_{\delta,\zeta_2}(\eta)}{2-\zeta_1-\frac \delta\alpha \zeta_2}\\
& \lesssim \inf_{\zeta_1,\zeta_2 \in \Sigma}   \frac{ \zeta_1^{-r_1}\delta^{\kappa(1-r_1)}+\zeta^{-r_2}\delta^{2-\kappa- r_2}}{2-\zeta_1-\frac \delta\alpha \zeta_2}
\end{align*}
where $\zeta_1 + \frac{\delta}\alpha \zeta_2 <2$ in $\Sigma$.
If we choose $\zeta_1=c<1$ and $\zeta_2=\frac{\alpha}{\delta}$ we can write 
\begin{align*}
D_R^{\mu_\alpha^\delta,\mu^\dagger}(u_\alpha^\delta,u^\dagger) 
& \lesssim \delta^{\kappa(1-r_1)}+ \delta^{2-\kappa(1-r_2)}
\end{align*}
where we need to assume $r_1<1$. 
The above convergence is optimized by $\kappa= \frac{2}{2+r_2-r_1}$ when $r_1\leq r_2$ in which case
\begin{align*}
D_R^{\mu_\alpha^\delta,\mu^\dagger}(u_\alpha^\delta,u^\dagger) 
& \lesssim \delta^\frac{2(1-r_1)}{2+r_2-r_1}.
\end{align*} 
If $r_1>r_2$ then we choose $\kappa=1$ which gives us convergence 
\begin{align*}
D_R^{\mu_\alpha^\delta,\mu^\dagger}(u_\alpha^\delta,u^\dagger) 
& \lesssim \delta^{1-r_1}.
\end{align*} 
\end{proof}

\begin{cor}
If we assume that $u^\dagger$ fulfills the source condition, that is, $r_1=0$ we get convergence
\begin{align*}
D_R^{\mu_\alpha^\delta,\mu^\dagger}(u_\alpha^\delta,u^\dagger) 
& \lesssim \delta^\frac{2}{2+r_2}.
\end{align*}
\end{cor}

\subsubsection*{Case $p> 2$}

\begin{thm}
Suppose that $X$ is a $p$-convex Banach space with some $p> 2$ and $R(u)=\frac{1}{p}\|u\|_X^p$. Moreover, suppose that Assumption \ref{ass:p>1_source} is satisfied with some orders $r_1, r_2\geq 0$ and $r_1<1$. For the choice $\alpha\simeq\delta^\kappa$, where 
\begin{equation*}
\kappa = 
\begin{cases}
\frac{2}{2+r_2-r_1} & {\rm for}\; r_1 \leq r_2\; {\rm and}\\
1 & {\rm for}\; r_2<r_1<1,
\end{cases}
\end{equation*}
we have convergence
\begin{equation*}
	D_R^{\mu_\alpha^\delta,\mu^\dagger}(u_\alpha^\delta,u^\dagger) 
\leq 
\begin{cases}
C_p\delta^{\frac{2(1-r_1)}{2+r_2-r_1}} & {\rm for}\; r_1 \leq r_2\; {\rm and}\\
C_p\delta^{1-r_1} & {\rm for}\; r_2<r_1.
\end{cases}
\end{equation*}
\end{thm}

\begin{proof}
We can give an alternative definition for the general Bregman distance by
\begin{align}\label{Bregman2}  
\begin{split}
D_R^{\mu_u}(u,v) &=\frac{1}{q}\|\mu_u\|_{X^*}^q 
-\langle\mu_u,v\rangle_{X^*\times X} +\frac{1}{p}\|v\|_X^p\\
&=(p-1)R(u)- \langle\mu_u,v\rangle_{X^*\times X} +R(v)
\end{split}
\end{align}
where $\mu_u\in\partial R(u).$
We get same kind of estimate for the Bregman distance as in \cite{bonesky2008minimization} 
\begin{align*}
D_R^{\mu_u}(u,v) &= \Big(1-\frac{1}{p}\Big)\|u\|_X^p-\langle\mu_u,v\rangle_{X^*\times X}+\frac{1}{p}\|v\|_X^p\\
&= \frac{1}{p}\|u-(u-v)\|_X^p-\frac{1}{p}\|u\|_X^p+\langle\mu_u,u-v\rangle_{X^*\times X}\\
&\geq \frac{C_p}{p}\|u-v\|_X^p.
\end{align*}
The last estimate above is given by the Xu-Roach inequalities \cite{xu1991}. 
The  Bregman distance given by (\ref{Bregman2}) coincides with our previous definition (\ref{Bregman1})
\begin{align*}
D_R^{\mu_u,\mu_v}(u,v) &= D_R^{\mu_u}(u,v)+D_R^{\mu_v}(v,u)\\ 
&= p(R(u)+R(v))+\langle\mu_u-\mu_v,u-v\rangle_{X^*\times X}-\|u\|_X^p-\|v\|_X^p\\
&= \langle\mu_u-\mu_v,u-v\rangle_{X^*\times X},
\end{align*}
for any $\mu_u\in\partial R(u)$ and $\mu_v\in\partial R(v)$.
Hence we get an estimate 
\begin{align*}
R(u-v)\leq C_pD_R^{\mu_u,\mu_v}(u,v).
\end{align*}
That is, \eqref{bregmanscaling} holds with $\theta = 1$ and $C_\theta(u,v)=C_p$. Hence when $p>2$ we get the same convergence rate as in case  $p=2$.
\end{proof}

\begin{rem}\label{Rem:comparison}
It is straightforward to see that polynomial decay of the distance function \cite{schusterbuch} implies an approximate source condition in Assumption \ref{ass:p>1_source}. Suppose we have
\begin{align*}  
d_\rho(\mu^\dagger) = \inf_{w\in Y}\{\|K^*w-\mu^\dagger\|_{X^*}\, |\,\ \|w\|_Y\leq \rho\}\leq \rho^{-k},
\end{align*}
where $k>0$.
This yields an estimate
\begin{align*}
e_{\alpha,\zeta}(\mu^\dagger)
& = \inf_{w\in Y}\Big\{\zeta R^\star\Big(\frac{1}{\zeta} (K^*w-\mu^\dagger)\Big)+\frac{\alpha}{2}\|w\|_Y^2\Big\}\\
& \simeq \inf_{\rho>0}\Big\{\zeta^{1-q}\rho^{-kq}+\delta^\kappa\rho^2\Big\}\\
& \simeq \delta^{\frac{kq\kappa}{kq+2}}\zeta^{-\frac{2(q-1)}{kq+2}}.
\end{align*}
Choosing $k=\frac{2(1-r_1)}{r_1q}$, where $r_1\in[0,1)$, we see that the last estimate above can be written
\begin{align*}
e_{\alpha,\zeta}(\mu^\dagger) & \simeq \delta^{\kappa(1-r_1)}\zeta^{-(q-1)r_1}
\end{align*}
which corresponds to the estimate given by Assumption \ref{ass:p>1_source} and (\ref{eq:p<2_estimates}).
\end{rem}

\subsection{Hilbert Space Embedding}

Since many estimates are crucially simplified by using Hilbert space structures, we discuss in the following an approach to obtain (possibly suboptimal) rates deduced from the results above using embedding.  We consider the case where $R$ is the $p$-th power of  a norm in a Banach space, with $p \geq 1$, and there exists a continuous embedding into a Hilbert space $X_0$. Indeed, we can assume the slightly weaker condition  
\begin{equation} \label{eq:Rembedding}
	R(u) \geq C \Vert u \Vert_{X_0}^p
\end{equation}
for all $u \in X.$
Note that by extending $R$ as infinite outside $X$ we can also state the same condition for arbitrary $u \in X_0$. 
Obviously the case $p=1$ is of particular interest here to cover e.g. total variation regularization (with the obvious embedding into $L^2$ for dimension less or equal two) and 
sparsity regularization (with the obvious embedding of $\ell^1$ into $\ell^2$).

In order to reduce to a Hilbert space framework, we assume that $K$ can be extended to $X_0$ and maps this space continuously to $Z$. Thus, $L=K^* K$ is a bounded self-adjoint operator on $X_0$ and thus has a spectral decomposition. In particular, we can formulate smoothness of a vector $\vartheta \in X_0$ with the condition
\begin{equation} \label{eq:Lsourcecondition}
	\vartheta = L^\mu \omega 
\end{equation}
for $\omega \in X_0$ and some $\mu \in (0,\frac{1}2)$.
We then use the relation $e_{\alpha,\zeta}(\vartheta) = - \inf_{v \in X} F_{\alpha,\zeta}(v;\vartheta)$ and estimate $F_{\alpha,\zeta}(v;\vartheta)$ from below. 
For this sake we use \eqref{eq:Rembedding} and \eqref{eq:Lsourcecondition} to get
\begin{eqnarray*} 
	F_{\alpha,\zeta}(v;\vartheta) & = & \frac{1}{2\alpha}\Vert K v \Vert_Y^2 - \langle \vartheta,v \rangle_{X^*\times X} + \zeta R(v) \\
	& \geq & 
	\frac{1}{2\alpha}\Vert L^\frac{1}2 v \Vert_{X_0}^2  - \Vert L^\mu v \Vert_{X_0} \Vert \omega \Vert_{X_0} +\zeta C \Vert v \Vert_{X_0}^p.
\end{eqnarray*}
Using the interpolation inequality 
$$ \Vert L^\mu v\Vert_{X_0} \leq  \Vert L^\frac{1}2 v\Vert_{X_0}^{2\mu} \Vert  v\Vert_{X_0}^{1-2\mu} $$
and Young's inequality we get estimate 
\begin{equation*}
	F_{\alpha,\zeta}(v;\vartheta)  \geq - C \Vert \omega \Vert_{X_0}^{\frac{p}{p-1+2\mu-p\mu}} \zeta^{-\frac{1-2\mu}{p-1+2\mu-p\mu}} \alpha^{\frac{p\mu}{p-1+2\mu-p\mu}}
\end{equation*}
for some constant $C$ independent of $v$, $\zeta$, and $\alpha$, which directly yields an upper bound for $e_{\alpha,\zeta}(\vartheta)$.
We mention that in the case $p=1$ we obtain
\begin{equation}
	\label{eq:embedding_estimate}
	e_{\alpha,\zeta}(\vartheta) \leq C \Vert \omega \Vert_{X_0}^{\frac{1}\mu}  \zeta^{-\frac{1-2\mu}{\mu}}\alpha .
\end{equation}

\section{Examples with random noise}
\label{sec:random_noise} 

\subsection{Frequentist framework} 

Let us recall that our work above towards unbounded noise was mostly motivated by random noise models, especially, the statistics of white noise. It is hence natural to reinterpret the results of Theorem \ref{deterministicestimates} as pointwise estimates for a random variable $U_\alpha^\delta$, which arises due to the randomness of the noise $N$. 
In the frequentist settings one is interested in the model
 \begin{align}\label{frequentist}
F^\delta = Ku^\dagger+\delta N,
\end{align}
where the data $F^\delta$ are generated by a deterministic true solution $u^\dagger$. In (\ref{frequentist}) the measurement $F^\delta=F^\delta(\omega)$ and the noise $N=N(\omega)$ are thought to be random variables. Here $\omega\in\Omega$ is an element of a complete probability space $(\Omega,\Sigma,\mathbb{P})$.

Following the idea in the earlier sections we consider a general frequentist risk denoted by $E_B$ between the estimator $U_\alpha^\delta=U_\alpha^\delta(\omega)$ and $u^\dagger$. Here, our error measure is given by the Bregman distance
\begin{equation}
\label{eq:freq_risk}
E_B(U_\alpha^\delta,u^\dagger)
 =\E \big(D_R^{\mu_\alpha^\delta,\mu^\dagger}(U_\alpha^\delta,u^\dagger)\big) .
\end{equation}
From the previous section we directly obtain a bound
\begin{multline*}
	E_B(U_\alpha^\delta,u^\dagger) 
	= \E\bigg\{\inf_{(\zeta_1,\zeta_2) \in (\mathbb{R}^+)^2} \left(\left(\zeta_1 + \frac{\delta}\alpha \zeta_2\right)^{1/(1-\theta)} C_\theta(U_\alpha^\delta ,u^\dagger)^{1/(1-\theta)} + \right. \\
\left.  
\frac{1}{1-\theta}
e_{\alpha,\zeta_1}(\mu^\dagger)+ \frac{\delta}{\alpha(1-\theta)} e_{\delta,\zeta_2}(K^*N)\right)\bigg\}. 
\end{multline*}
A canonical example of frequentist risk \eqref{eq:freq_risk} is the mean integrated squared error (MISE)
\begin{align*}
E_B(U_\alpha^\delta,u^\dagger)=\E\|U_\alpha^\delta-u^\dagger\|_{X}^2,
\end{align*}
where a quadratic regularization term $R(u) = \norm{u}_X^2$ is assumed. Convergence rates of MISE have been widely studied in the literature, see \cite{Cavalier2008,cavalier2002sharp}.  
 
We observe that a finite estimate can only be obtained if 	$\E(e_{\delta,\zeta }(K^*N)) < \infty$ at least for some $\zeta > 0$. Under the typical choices of $R$ the finiteness for any $\delta$ and $\zeta$ is obtained if 
\begin{equation*}
	\E(e_{1,1}(K^*N)) < \infty.
\end{equation*}
This condition can be interpreted as an abstract smoothing condition for the operator $K$, as we shall see it can be identified with $K$ being a trace-class operator. 

In order to choose optimal parameters we first have to clarify which of them are random. Since $\zeta_2$ is an auxiliary parameter appearing in the estimates only, not affecting any computation, it can be optimized in dependence of $K^*N$ and hence it also becomes a random variable. The situation is less obvious with respect to $\alpha$. Indeed it turns out that the question is exactly related to the issue of a-priori vs. a-posteriori parameter choice in the deterministic setup (cf. \cite{engl1996regularization}). The {\em a-priori parameter choice} $\alpha=\alpha(\delta)$ leads to a parameter independent of the realization of the noise $N$, while the {\em a-posteriori parameter choice} $\alpha=\alpha(\delta,F)$ makes the parameter a random variable of $N$.  
Since the specific choices of $\alpha$ rely on the form of the regularization functional, we shall further investigate the general risk \eqref{eq:freq_risk} in three very prominent cases, the classical one of Tikhonov regularisation (two-homogeneous $R$), the more general regularisation with Besov penalty and the popular total variation regularisation.

\subsection{Gaussian case}

Let us review the implications of our results in the canonical special case of a squared norm based regularization penalty $R(u) = \frac 12 \norm{u}_X^2$ for $X=Y=L^2(\T^d)$. We assume that $N$ is generalized white noise statistics in ${\mathcal D}'(\T^d)$, that is, we have $\E N = 0$ and 
\begin{equation*}
	\E \langle N, \phi\rangle_{\mathcal{D}'\times \mathcal{D}} \langle N, \psi\rangle_{\mathcal{D}'\times \mathcal{D}} = \langle \phi,\psi\rangle_{\mathcal{D}'\times \mathcal{D}}
\end{equation*} 
for any test functions $\phi,\psi\in C^\infty(\T^d)$, where $\langle \cdot, \cdot\rangle_{\mathcal{D}'\times \mathcal{D}}$ denotes the duality pairing.
It is well-known that the realizations of $N$ belong to $Z^*=H^{-d/2-\epsilon}(\T^d)$ almost surely for any $\epsilon>0$. For a sharp result, see \cite{Veraar11}. We want to concentrate on the phenomena appearing due to large noise and hence assume an exact source condition for the true unknown $u^\dagger$  in the following.

In this example, two factors simplify our analysis remarkably. First, the symmetric Bregman distance coincides with the squared norm (as discussed in Example \ref{ex:two-homogeneous})
\begin{equation*}
	D_R^{\mu_u,\mu_v}(u,v) = \norm{u-v}^2_{L^2(\T^d)}.
\end{equation*}
Secondly, the term $e_{\alpha,\zeta}(K^*N)$ can be explicitly estimated since
\begin{eqnarray}
	\label{eq:gaussian_e_term}
	e_{\alpha,\zeta}(K^*N) & = & \inf_{w\in L^2(\T^d)} \left(\zeta R^\star\left(\frac{K^*W-K^*N}{\zeta}\right)+ \frac \alpha 2 \norm{W}^2_{L^2(\T^d)}\right) \nonumber \\
	& = & \frac{1}{2\zeta} \inf_{w\in L^2(\T^d)} \left( \norm{K^*W-K^*N}^2_{L^2(\T^d)} + \alpha \zeta \norm{W}^2_{L^2(\T^d)}\right).
\end{eqnarray}
Let us record the following short calculation as a lemma. For precise notation, let use denote ${\mathcal R}_\beta = (K^*K+\beta I)^{-1}: L^2(\T^d)\to L^2(\T^d)$, $\beta > 0$, to highlight the restriction of $K^*$ (and $K$) to $X=Y=L^2(\T^d)$.

\begin{lemma}
Consider $K$ as a bounded linear operator $K : L^2(\T^d) \to H^{t}(\T^d)$ for $t>d/2$. Then it follows that
\begin{equation*}
	e_{\alpha,\zeta}(K^*N) = \frac{\alpha}{2} 
	\langle N, K {\mathcal R}_{\alpha\zeta}K^* N\rangle_{H^{-t}(\T^d)\times H^t(\T^d)}
\end{equation*}
and
\begin{equation}
	\label{eq:gaussian_expec_e}
	\E e_{\alpha,\zeta}(K^*N) = \frac \alpha 2{\rm Tr}_{L^2(\T^d)} ( K{\mathcal R}_{\alpha\zeta}K^*).
\end{equation}
\end{lemma}

\begin{proof}
The minimizing estimator of problem \eqref{eq:gaussian_e_term} is given by
$W_{\alpha\zeta}=K{\mathcal R}_{\alpha\zeta}K^*N$.
Hence we can write
\begin{equation*}
\|K^*W_{\alpha\zeta}-K^*N \|_{L^2(\T^d)}^2+\alpha\zeta\|W_{\alpha\zeta}\|_{L^2(\T^d)}^2 
=  \alpha \zeta\langle N, K{\mathcal R}_{\alpha\zeta}K^* N \rangle_{H^{-t}(\T^d)\times H^t(\T^d)},
\end{equation*}
where $t>d/2$
It is well-known that $N$ as white noise has a series representation $N = \sum_{j=1}^\infty N_j \psi_j$ almost surely, where $N_j \sim {\mathcal N}(0,1)$ are i.i.d. and $\{\psi_j\}_{j=1}^\infty$ constitutes any orthonormal basis of $L^2(\T^d)$. The claim \eqref{eq:gaussian_expec_e} now follows easily by applying the series representation together with independence of $N_i$ and $N_j$ for $i\neq j$.
\end{proof}

Let us mention that the quantity on the right-hand side of the estimate \eqref{eq:gaussian_expec_e}, 
$$ {\rm Tr}_{L^2(\T^d)} ( K{\mathcal R}_{\alpha\zeta}K^*) = {\rm Tr}_{L^2(\T^d)} ( (KK^*+{\alpha\zeta} I)^{-1} KK^*), $$
is known as the {\em effective dimension} in literature (cf. \cite{zhang2005learning}). In the finite dimensional case it is between zero (as $\alpha \zeta \rightarrow \infty$) and the rank of $KK^*$ (as $\alpha \zeta \rightarrow 0$). In the following we use a conservative estimate of the effective dimension in order to illustrate the results, optimal estimates can be achieved under special assumptions, which is beyond our scope (cf. e.g. \cite{lu2014discrepancy}). Our analysis in the non-Gaussian case indicates that $\E e_{\alpha,\zeta}(K^*N)$ is the basis for understanding a generalization of effective dimension for such, its analysis is a possibly important question for future research.

\begin{thm}\label{thm:gaussian}
Assume that $K : L^2(\T^d) \to H^{t}(\T^d)$, where $t>d/2$, is a Hilbert--Schmidt operator in $L^2(\T^d)$ and $R(u) = \frac 12 \norm{u}^2_{L^2(\T^d)}$. When the true unknown $u^\dagger$ fulfills the exact source condition $\mu^\dagger=K^*w^\dagger$, where $w^\dagger \in L^2$, we obtain the convergence rate
\begin{equation*}
E_B(U_\alpha^\delta,u^\dagger) = \E \norm{U_\alpha^\delta-u^\dagger}^2_{L^2(\T^d)} \lesssim \delta^{2/3},
\end{equation*}
with choice $\alpha \simeq \delta^{2/3}$.
\end{thm}

\begin{proof}
Considering \eqref{optimality} where we have now $u=\mu\in\partial R(u)$. Therefore, we can write 
\begin{equation*}
K^*(KU_\alpha^\delta-(f+\delta N))+\alpha U_\alpha^\delta=0
\end{equation*}
and consequently
\begin{equation}
\label{eq:gaussian_breg_aux1}
K^*(KU_\alpha^\delta-Ku^\dagger)+\alpha (U_\alpha^\delta-u^\dagger)=\delta K^*N-\alpha K^*w^\dagger
\end{equation}
where  $u^\dagger=\mu^\dagger=K^*w^\dagger$ for $w^\dagger\in L^2(\T^d)$. Taking duality product of $U_\alpha^\delta-u^\dagger$ and equation \eqref{eq:gaussian_breg_aux1} yields
\begin{equation}
\label{eq:gaussian_breg_aux2}
\|KU_\alpha^\delta-Ku^\dagger\|_{L^2(\T^d)}^2+\alpha\|U_\alpha^\delta-u^\dagger\|_{L^2(\T^d)}^2 = \delta \langle K^*N, U_\alpha^\delta-u^\dagger\rangle_{L^2(\T^d)}+\alpha\langle w^\dagger,K(u^\dagger-U_\alpha^\delta)\rangle_{L^2(\T^d)}.
\end{equation}
We will approximate the right hand side terms separately. 
Following the idea behind the estimate \eqref{eq:bregman_dist_est_first} we bound the first term on the right hand side of \eqref{eq:gaussian_breg_aux2} by
\begin{multline*}
\label{eq:gaussian_breg_aux3}
\delta \langle K^*N, U_\alpha^\delta-u^\dagger\rangle_{L^2(\T^d)} 
= \delta\langle K^*N-K^*W, U_\alpha^\delta-u^\dagger\rangle_{L^2(\T^d)}+\delta \langle W, K(U_\alpha^\delta-u^\dagger)\rangle_{L^2(\T^d)} \\
 \leq  \frac{\delta\zeta_2}{2}\|K^*N-K^*W\|_{L^2(\T^d)}^2+\frac{\delta}{2\zeta_2}\|U_\alpha^\delta-u^\dagger\|_{L^2(\T^d)}^2+\frac{\delta^2}{2}\|W\|_{L^2(\T^d)} ^2+\frac{1}{2}\|K(U_\alpha^\delta-u^\dagger)\|_{L^2(\T^d)}^2
\end{multline*}
for any $\zeta_2>0$.
For the last term in \eqref{eq:gaussian_breg_aux2} we have 
\begin{align*}
\alpha\langle w^\dagger,K(u^\dagger-U_\alpha^\delta)\rangle_{L^2(\T^d)} 
& \leq \frac{\alpha^2}{2}\|w^\dagger\|_{L^2(\T^d)}^2+\frac{1}{2}\|K(U_\alpha^\delta-u^\dagger)\|_{L^2(\T^d)}^2.
\end{align*}
Since $w=W(\omega)\in L^2(\T^d)$ is arbitrary, using the estimates above we get
\begin{eqnarray}
\label{eq:gaussian_breg_est}
\|U_\alpha^\delta-u^\dagger\|_{L^2(\T^d)}^2 & \leq & \frac{1}{\alpha-\frac{\delta}{2\zeta_2}} \left\{\inf_{w\in L^2(\T^d)} \left(\frac{\delta \zeta_2}{2}\|K^*N-K^*w\|_{L^2(\T^d)}^2+\frac{\delta^2}{2}\|w\|_{L^2(\T^d)}^2\right) +\frac{\alpha^2}2\|w^\dagger\|_{L^2(\T^d)}^2\right\} \nonumber \\
& \leq & \frac{2\delta}\alpha e_{\delta,\frac \alpha \delta}(K^*N) + \alpha \|w^\dagger\|_{L^2(\T^d)}^2.
\end{eqnarray}
Above, we obtained the last estimate by choosing $\zeta_2 = \frac{\delta}{\alpha}$.

In order to derive a convergence rate we point out that ${\mathcal R}_\beta$ is a self-adjoint semipositive-definite bounded linear operator satisfying $\norm{{\mathcal R}_\beta^{1/2}}_{L^2(\T^d)\to L^2(\T^d)} \leq \frac 1{\sqrt{\beta}}$ and consequently
\begin{equation}
\label{eq:gaussian_breg_aux4}
\E e_{\delta,\frac \alpha \delta}(K^*N)=\frac \delta 2{\rm Tr}_{L^2(\T^d)}(K{\mathcal R}_{\alpha}K^*)
\leq \frac \delta{2\alpha}{\rm Tr}_{L^2(\T^d)}(K K^*).
\end{equation}
Since $K^*$ is a Hilbert-Schmidt operator, we have ${\rm Tr}_{L^2(\T^d)}(KK^*)<\infty$.

Now it follows from equations \eqref{eq:gaussian_breg_est} and \eqref{eq:gaussian_breg_aux4} that 
\begin{equation}
	\label{eq:gaussian_breg_est2} 
\E\|U_\alpha^\delta-u^\dagger\|_{L^2(\T^d)}^2\leq\frac{\delta^2}{\alpha^2}{\rm Tr}_{L^2(\T^d)}(KK^*)+\alpha\|w^\dagger\|_{L^2(\T^d)}^2.
\end{equation}
The bound in \eqref{eq:gaussian_breg_est2} is optimised by choosing $\alpha \simeq \delta^{2/3}$, which also yields the claim.
\end{proof}

From the previous theorem we see that the assumption of finite trace of $KK^*: L^2(\T^d) \to L^2(\T^d)$ is indeed equivalent to the condition
\begin{equation}
	\E(e_{\delta,\gamma}(K^*N)) < \infty 
\end{equation}
for some $\delta, \gamma>0$ 
as well as to the condition
\begin{equation}  
\E(	\Vert K^*N \Vert_{L^2(\T^d)}^2) < \infty, 
\end{equation}
which appears to be a natural requirement.


One can observe better convergence rates if faster decay of eigenvalues of $KK^*$ is assumed.

\begin{thm}\label{thm:smoother}
Suppose that $R(u) = \frac 12 \norm{u}^2_{L^2(\T^d)}$. Moreover, assume that $\{\lambda_j\}_{j=1}^\infty$ are eigenvalues of $KK^*:L^2(\T^d) \to L^2(\T^d)$ and there exists $0<m\leq 1$ such that
\begin{equation*}
	\sum_{j=1}^\infty \lambda_j^{m} < \infty.
\end{equation*}
Then, when the true unknown $u^\dagger$ fulfills the exact source condition $\mu^\dagger=K^*w^\dagger$, where $w^\dagger \in L^2$, it follows that for $\alpha\simeq\delta^\kappa$, where $\kappa=\frac{2}{2+m}$, we obtain
\begin{align*}
E_B(U_\alpha^\delta,u^\dagger)=\E\norm {U_\alpha^\delta-u^\dagger}_{L^2(\T^d)}^2\lesssim \delta^{\frac{2}{2+m}}.
\end{align*}
\end{thm}

\begin{proof}
Suppose $p$ and $q$ are H\"older conjugates such that $m = \frac 1q$.
By applying Young's inequality to
$$p\alpha^{p/q}\lambda_j\leq \alpha^{p}+\lambda_j^p\leq(\alpha+\lambda_j)^p, \quad \alpha,\lambda_j\geq 0, $$ we obtain
\begin{align}\label{Young}
(p\alpha^{p/q}\lambda_j)^{1/p}\leq \alpha+\lambda_j.
\end{align}
This yields
\begin{equation*}
Tr_{L^2(\T^d)}(K{\mathcal R}_\alpha K^*)= \sum_{j=1}^\infty \frac{\lambda_j}{\lambda_j+\alpha}
\leq \sum_{j=1}^\infty \frac{\lambda_j}{(p\alpha^{p/q}\lambda_j)^{1/p}}
\leq \frac{1}{p^{1/p}\alpha^{1/q}}\sum_{j=1}^\infty\lambda_j^{1/q}.
\end{equation*}
Therefore, by equation \eqref{eq:gaussian_breg_est} the frequentist risk is bounded by
\begin{align*}
\E\|U_\alpha^\delta-u^\dagger\|_{L^2(\T^d)}^2\leq \frac{\delta^2}{p^{1/p}\alpha^{1/q+1}}\sum_{j=1}^\infty\lambda_j^{1/q}+\alpha\|w^\dagger\|_{L^2(\T^d)}^2.
\end{align*}
The proof is concluded by optimizing $\alpha\simeq\delta^\kappa$.
\end{proof}

As mentioned before the mean integrated squared error (MISE) of an estimator $U^\delta_\alpha$ is defined 
\begin{align}\label{MISE}
R(U^\delta_\alpha,u^\dagger)=\E\|U^\delta_\alpha-u^\dagger\|_{L^2}^2(\T^d).
\end{align}
The minimax risk $r_\delta(H^{r}(\T^d),M)$ on the Sobolev space $H^{r}(\T^d)$ is then given by 
\begin{align*}
r_\delta(H^r(\T^d),M)=\inf_{U^\delta_\alpha}\sup_{\|u^\dagger\|_{H^r(\T^d)}<M}R(U^\delta_\alpha,u^\dagger),
\end{align*}
where the infimum is taken over all estimators of the form $U^\delta_\alpha=g(F^\delta)$. Here we have denoted $g\in \mathcal{B}(H^{-d/2-\epsilon}(\T^d) ,H^r(\T^d))$ where $\mathcal{B}(H^{-d/2-\epsilon}(\T^d),H^r(\T^d))$ is the set of Borel measurable functions from $H^{-d/2-\epsilon}(\T^d)$ to $H^r(\T^d)$. Next we will compare the convergence results of Theorems \ref{thm:gaussian} and \ref{thm:smoother} to the known minimax convergence rates for same problems. 

\begin{rem}
As an example of a group of operators that fills the conditions in Theorem \ref{thm:smoother} we can take bijective elliptic pseudodifferential operators that are $t>\frac{d}{2m}$ (where $m=1$ in the case described in Theorem \ref{thm:gaussian}) orders smoothing 
 e.g. $K=(I-\Delta)^{-\frac{t}{2}}$. We assume the exact source condition in the Theorems \ref{thm:gaussian} and \ref{thm:smoother}, that is, $u^\dagger=\mu^\dagger=K^*w^\dagger$, where $w^\dagger\in L^2$, hence we can conclude $u^\dagger \in H^r(\T^d)$, where $r=t$. This means we assume the minimum extra smoothness from $u^\dagger$, that is, the smoothness of the unknown and the order of smoothing of the forward operator are the same.
 
Since $r=t$ we can rewrite the convergence rate $\kappa$ in form 
\begin{align*}
\kappa = \frac{2}{2+m} = \frac{2r}{r+t+tm}.
\end{align*}
Note that $tm=d/2+\epsilon$ and hence the convergence rates achieved in Theorems \ref{thm:gaussian} and \ref{thm:smoother} agree, up to $\epsilon>0$ arbitrarily small, with the minimax convergence rate, see e.g. \cite{Cavalier2008, Hohage2016}. 
\end{rem}

\subsection{Besov penalty}

Suppose that functions $\{\psi_\ell\}_{\ell=1}^\infty$ form
an orthonormal wavelet basis 
for $L^2(\T)$ on the one-dimensional torus $\T$, where we have utilized global indexing.
We can characterize the periodic Besov space $B^s_{pq}(\T)$ using the given basis
in the following way: the series
\begin{equation*}
	u(x) = \sum_{\ell=1}^\infty u_\ell \psi_\ell(x)
\end{equation*}
belongs to $B^s_{pq}(\T)$ if and only if
\begin{equation}
	\label{eq:orig_norm}
	2^{js} 2^{j(\frac 12 -\frac 1p)} \left( \sum_{\ell=2^j}^{2^{j+1}-1} |u_\ell|^p\right)^{1/p} \in \ell^q(\N).
\end{equation}
We assume that the basis is $r$-regular for $r$ large enough in order to provide a basis for a Besov space with smoothness $s$ \cite{Daubechies}. 
Here we are concerned with the special case $p=q$ and use abbreviation $B^s_{p} = B^s_{pp}$. It is well-known that
an equivalent norm to \eqref{eq:orig_norm} is given by
\begin{equation*}
	\left\|\sum_{\ell=1}^\infty u_\ell\psi_\ell\right\|_{B^s_{p}(\T)}
	= \left(\sum_{\ell=1}^\infty \ell^{p(s+\frac 12)-1} |u_\ell|^p\right)^{1/p}.
\end{equation*}

\subsubsection{Case $p=1$}
Suppose that $X=Y = L^2(\T)$ with orthogonal basis $\{\psi_\ell\}_\ell$. The noise $N$ is assumed to have same statistics as in previous section. Here, we consider a regularization term given by
\begin{equation}\label{eq:R_besov}
	R(u) = \norm{u}_{B_1^s(\T)} = \sum_{\ell=1}^\infty \ell^{s-1/2} |u_\ell|
\end{equation}
for $s\geq \frac 12$, where $u= \sum_{\ell=1}^\infty u_\ell \psi_\ell$.  
The dual norm in \eqref{eq:Sdefinition} is simply the norm of $B^{-s}_\infty(\T)$, i.e.,
\begin{equation*}
	S(q) = \norm{q}_{B^{-s}_\infty(\T)} = \sup_{\ell\in\N} \ell^{1/2-s} |q_\ell|
\end{equation*}
for $q = \sum_{\ell=1}^\infty q_\ell \psi_\ell \in B^{-s}_\infty(\T)$. Notice carefully that for parameters $s\geq \frac 12$ the functional $R$ satisfies conditions (R1)-(R4) with the weak topology, since there is a continuous embedding from $B^s_1(\T)$ to $L^2(\T)$. 

We notice that an arbitrary approximate source condition of type \eqref{eq:1source} requires a sparse structure of the true unknown as pointed out by the following lemma. Therefore, 
it does not cover a general class of unknowns for this 1-homogeneous example. 

\begin{lemma}
Let us assume that $R$ is given by \eqref{eq:R_besov} and $K:L^2(\T)\to L^2(\T)$ is such that for all $\ell\in \mathbb{N}$ there exists a function $w^{(\ell)}\in Y$ such that 
\begin{align}\label{eq:diagAssumption}
K^*w^{(\ell)} = e^{(\ell)}, 
\end{align} 
where $e^{(\ell)}=(0,\dots,0,1,0,\dots)$ is the infinite unit sequence with $1$
at the $\ell$-th position and $0$ else. 
Then the following two statements are equivalent:
\begin{itemize}
	\item[(i)] The subgradient $\mu^\dagger$ satisfies the approximate source condition in Assumption \ref{assumption:p=1} with some $r_1\geq 0$.
	\item[(ii)] The unknown $u^\dagger$ is non-zero only in a finite set of coefficient.
\end{itemize}
\end{lemma}
\begin{proof}
Since $R$ is defined as in \eqref{eq:R_besov} we see that the the subgradient  is given by
\begin{align*}
\mu^\dagger_\ell =
  \begin{cases}
    \ell^{s-\frac{1}{2}}   & \quad \text{when}\, u^\dagger_\ell>0 \\
   -\ell^{s-\frac{1}{2}}  & \quad \text{when}\, u^\dagger_\ell<0 \\
  \end{cases}
\end{align*}
and $\mu^\dagger_\ell\in (-\ell^{s-\frac{1}{2}},\ell^{s-\frac{1}{2}})$ when $u^\dagger_\ell=0$. 

Since $K^*:L^2\to L^2$ we have $(K^*w)_\ell\to0$, when $\ell\to\infty$. If $u^\dagger$ has infinitely many non-zero coefficients 
\begin{equation*}
	S(\mu^\dagger-K^*w) = \sup_{\ell\in\N} \ell^{1/2-s} |\mu^\dagger_\ell-(K^*w)_\ell|\geq 1
\end{equation*}
and the approximate source condition \eqref{eq:1source} can not be satisfied. On the other hand if the unknown is non-zero only in a finite set of coefficient, that is, there exists such $L$ that $u^\dagger_\ell=0$ for $\ell>L$ we can choose subgradient $\mu^\dagger$ so that $\mu_\ell^\dagger=0$ for $\ell>L$.  Using assumption \eqref{eq:diagAssumption} we can then choose $w\in L^2$ so that $(K^*w)_\ell=\mu^\dagger_\ell$. Hence  
$S(\mu^\dagger-K^*w)=0$ and the source condition \eqref{eq:1source} is fulfilled with any $r_1\geq 0$. 
\end{proof}

As an example of group of operators that satisfies assumption  \eqref{eq:diagAssumption} we can take operators with diagonal structure. The more general meaning of the assumption \eqref{eq:diagAssumption} has been studied for example in \cite{Flemming2015} and the references therein. 
 
In addition to the above, we make an assumption on the smoothness of $K$ and $K^*$ by requiring that there exists a constant $C>0$ and $t >\frac 12$ such that both satisfy
\begin{equation}
	\label{eq:besov_smooth_cond}
	\frac 1C\norm{\psi}_{B_2^r} \leq \norm{K\psi}_{B_2^{t+r}} \leq C\norm{\psi}_{B_2^r}
\end{equation}
(similar for $K^*$) 
for $r\in \R$ and $\psi \in B^r_2(\T)$.  
The above smoothness condition enables a straightforward study of the noise terms, which allows us to deduce contraction rates.

Under the given assumptions we can write (recall equation \eqref{eq:one_homog_edelta})
\begin{equation*}
	e_{\delta,\zeta}(\eta) = \frac \delta 2 \inf_{w \in {\mathcal W}} \norm{w}^2_Y,
\end{equation*}  
with a fixed realization $\eta=K^*n=K^*N(\omega)$, where 
\begin{equation*}
	{\mathcal W} = \left\{w\in L^2(\T) \; \bigg| \; \sup_{\ell \in\N} \;\ell^{1/2-s} \left|\langle n-w,K\psi_\ell\rangle_{B^{-t}_2\times B^t_2}\right| \leq \zeta \right\}.
\end{equation*}

\begin{lemma}
\label{lem:besov}
Let us assume that $K: L^2(\T) \to L^2(\T)$ satisfies condition \eqref{eq:besov_smooth_cond} with a parameter $t>\frac 12$ and $R$ is defined by \eqref{eq:R_besov} for $s\geq \frac 12$. Then it holds that 
\begin{equation*}
	\E e_{\delta,\zeta}(K^*N) \lesssim \delta \zeta^{-\frac{2}{2s+2t-1}}.
\end{equation*}
\end{lemma}

\begin{proof}
From condition \eqref{eq:besov_smooth_cond} it follows that 
\begin{equation*}
	e_{\delta,\zeta}(\eta) = \frac \delta 2 \inf_{w \in {\mathcal W}} \norm{w}^2_{L^2}
	\leq \frac{C \delta} 2 \inf_{w \in {\mathcal W}} \norm{K^*w}^2_{B^t_2}
\end{equation*}
almost surely.
The condition $w\in {\mathcal W}$ for a fixed realization $n=N(\omega)$ is equivalent to
\begin{equation*}
	|(K^*n)_\ell - (K^*w)_\ell| \leq \zeta \ell^{s-\frac 12} =: D_\ell
\end{equation*}
uniformly for all $\ell \in \N$
and hence
\begin{equation*}
	\inf_{w \in {\mathcal W}}\norm{K^*w}^2_{B^t_2} = \sum_{\ell=1}^\infty \inf_{|(K^*n)_\ell - (K^*w)_\ell| \leq D_\ell} \ell^{2t}|(K^*w)_\ell|^2 \leq \sum_{\ell=1}^\infty \ell^{2t} \max (|(K^*n)_\ell|-D_\ell, 0)^2.
\end{equation*}
Now $(K^* N)_\ell = \langle K^* N, \psi_\ell\rangle$ is a normally distributed random variable with zero mean and variance $\sigma_\ell^2 = \norm{K \psi_\ell}^2_{L^2} \simeq \norm{\psi_\ell}^2_{B^{-t}_2}\simeq \ell^{-2t}$ (according to \eqref{eq:besov_smooth_cond}). Therefore, we have
\begin{eqnarray*}
	\E \inf_{w \in {\mathcal W}}\norm{w}^2_{L^2}
	 & \leq & \E \sum_{\ell=1}^\infty \ell^{2t} \max (|(K^*N)_\ell|-D_\ell, 0)^2 \\
	& \lesssim & \sum_{\ell=1}^\infty \frac{\ell^{2t}}{\sigma_\ell} \int_{D_\ell}^\infty (x-D_\ell)^2 \exp\left(-\frac{x^2}{2\sigma_\ell^2}\right) dx \\
	& \leq & \sum_{\ell=1}^\infty \frac{\ell^{2t}}{\sigma_\ell} \int_0^\infty x^2 \exp\left(-\frac{x^2}{2\sigma_\ell^2}\right) dx 
	\cdot \exp\left(-\frac{D_\ell^2}{2\sigma_\ell^2}\right) \\
	& \simeq & \sum_{\ell=1}^\infty \sigma_\ell^2 \ell^{2t} \exp\left(-\frac{D_\ell^2}{2\sigma_\ell^2}\right)
	\simeq \sum_{\ell=1}^\infty \exp\left(-\frac 12 \zeta^2 \ell^{2s+2t-1}\right).
\end{eqnarray*}
Due to our assumptions on $s$ and $t$ we notice that the last sum converges.
The sum above can be approximated as follows
\begin{eqnarray*}
	 \sum_{\ell=1}^\infty \exp\left(-\frac 12 \zeta^2 \ell^{2s+2t-1}\right) 
	 & \simeq & \int_0^\infty \exp\left(-\frac 12 \left(\zeta^{\frac 2{2s+2t-1}} x\right)^{2s+2t-1}\right)dx \\
	 & = & \zeta^{\frac{-2}{2s+2t-1}} \int_0^\infty \exp\left(-\frac 12 y^{2s+2t-1}\right)dy
	 \simeq \zeta^{\frac{-2}{2s+2t-1}},
\end{eqnarray*}
where we applied a change of variable $y = \zeta^{\frac 2{2s+2t-1}} x$.
This yields the claim.
\end{proof}

\begin{thm} 
\label{thm:wavelet}
Let us assume that $K: L^2(\T) \to L^2(\T)$ satisfies conditions \eqref{eq:diagAssumption} and \eqref{eq:besov_smooth_cond} with parameter $t>\frac 12$, $R$ is defined by \eqref{eq:R_besov} for $s\geq \frac 12$ and $u^\dagger$ is supported on a finite number of coefficients. Then $\mu^\dagger$ satisfies an approximate source condition in Assumption \ref{assumption:p=1} with any $r_1\geq 0 $. For the choice $\alpha \simeq \delta^\kappa$, where
\begin{equation*}
	\kappa =  \frac{2s+2t}{2s+2t+1},
\end{equation*}
we obtain the convergence rate
\begin{equation*}
	\E D_R^{\mu_\alpha^\delta,\mu^\dagger}(U_\alpha^\delta,u^\dagger) \lesssim \delta^\kappa.
\end{equation*}
\end{thm}

\begin{proof}
First, in equation \eqref{eq:one_homog_D_R_est} we apply
\begin{eqnarray}
	\label{eq:besov_proof_1homog}
	\E D_R^{\mu_\alpha^\delta,\mu^\dagger}(U_\alpha^\delta,u^\dagger)
	& \leq &
	\inf_{\zeta_1\in\R_+}\left(\zeta_1 \left(1 + \frac{\delta}{\alpha} \E e_{\delta, \frac{\alpha\gamma}{\delta}}(K^*N)\right) 
	+  e_{\alpha,\zeta_1}(\mu^\dagger)\right) \nonumber \\
	& & + \inf_{\zeta_2\in\R_+}\left(
	\frac \delta \alpha \zeta_2  \left(1 + \frac{\delta}{\alpha} \E e_{\delta, \frac{\alpha\gamma}{\delta}}(K^*N)\right)
	+\frac{\delta}{\alpha} \E e_{\delta,\zeta_2}(K^* N)\right) \nonumber \\
	& =: & \widetilde M_1 + \widetilde M_2.
\end{eqnarray}
Notice that by Lemma \ref{lem:besov} and assumption $\alpha = \delta^\kappa$, $\kappa\leq 1$, we have
\begin{equation*}
	\E e_{\delta, \frac{\alpha\gamma}{\delta}}(K^*N) \lesssim \gamma^{-s'} \delta^{1+(1-\kappa)s'} \lesssim 1
\end{equation*}
for a constant $\gamma$, where we denote $s' = \frac{2}{2s+2t-1}>0$ for convenience. Therefore, we see as in the proof of Theorem \ref{onehomthm1} that
\begin{equation*}
	\widetilde M_1 \lesssim \inf_{\zeta_1\in\R_+} \left\{\zeta_1 + \delta^\kappa \zeta_1^{-r_1}\right\}
	\simeq \delta^{\frac{\kappa}{1+r_1}}
\end{equation*}
and
\begin{equation*}
	\widetilde M_2 \lesssim \inf_{\zeta_2\in\R_+} \left\{\delta^{1-\kappa} \zeta_2 + \delta^{2-\kappa} \zeta_2^{-s'} \right\} \simeq \delta^{(1-\kappa)\frac{s'}{1+s'} + (2-\kappa) \frac{1}{s'+1}} = \delta^{\frac{(1-\kappa)s'+2-\kappa}{1+s'}}.
\end{equation*}
The convergence rate is minimized for $\kappa$ which satisfies
\begin{equation*}
	\frac{\kappa}{1+r_1} = \frac{(1-\kappa)s'+2-\kappa}{1+s'}.
\end{equation*}
Since $r_1$ can be chosen arbitrarily small, we conclude that 
\begin{equation*}
	\kappa = \frac 12\cdot \frac{2+s'}{1+s'} = \frac{2s+2t}{2s+2t+1}. 
\end{equation*}
This concludes the proof.
\end{proof}
%

\subsubsection{Case $1<p\leq 2$}

Let us set $X=L^2(\T)$.
We consider here the special case when the forward operator $K$ in \eqref{eq:main_eq} can be diagonalized in the basis $\{\phi_\ell\}_\ell$, i.e., $\langle \phi_{\ell}, K \phi_{\ell'} \rangle = 0$, whenever $\ell\neq \ell'$. It follows that we can reduce our model to a countable number of independent equations
\begin{equation*}
	f_\ell = k_\ell u_\ell + \delta N_\ell
\end{equation*}
for $\ell\in \N$, where $f_\ell = \langle f, \phi_\ell\rangle$, $u_\ell = \langle u, \phi_\ell\rangle$, $k_\ell = \langle \phi_\ell, K \phi_\ell\rangle$ and the random variables $N_\ell = \langle N, \phi_\ell\rangle$ are normally distributed i.i.d.
Similar to the case $p=1$ we assume that $K$ satisfies \eqref{eq:besov_smooth_cond}, which corresponds to assuming 
$k_\ell \simeq \ell^{-t}$
asymptotically with respect to $\ell$.

Suppose that the regularization functional $R$ is given by $R(u) 	= \frac{1}{p}\|u\|^p_{B^s_p(\T)}$
for $1<p<2$ and $s\geq\frac 1p - \frac 12$ so that $B^s_p(\T)$ can be embedded continuously to $X$. The convex conjugate $R^\star$ satisfies 
\begin{equation*}
	R^\star(u)  = \frac 1q \sum_{\ell=1}^\infty \ell^{q(-s+\frac 12)-1} |u_\ell|^q
	= \frac{1}{q}\|u\|^q_{B^{-s}_q(\T)},
\end{equation*}
where $p$ and $q$ are H{\"o}lder conjugates. 

For the convenience of the reader, we assume that $\mu^\dagger$ satisfies the accurate source condition, i.e.,
\begin{equation*}
	\mu^\dagger = K^* w.
\end{equation*}
We can then write
\begin{equation}  
\label{eq:p>1_first_breg_estimate}
\E D_R^{\mu_\alpha^\delta,\mu^\dagger}(u_\alpha^\delta,u^\dagger) 
\lesssim \E M_1+\E M_2
\end{equation}
where $\E M_1=\alpha$ and 
\begin{align*}
M_2= \inf_{\zeta_2 \in \R_+}\left(\zeta_2^{\frac 2 {2-p}} \left(\frac{\delta}\alpha\right)^{\frac 2 {2-p}} \bigg(1 + \frac{\delta}{\alpha}e_{\frac{\delta}{2},\frac{\alpha\gamma}{\delta}}(K^*N)\bigg)+ \frac{\delta}{\alpha} e_{\delta,\zeta_2}(K^* N)\right).
\end{align*}

\begin{lemma}
\label{lem:besov_p_e_estimate}
Let us assume that $K$ and $R$ are as above, $s\geq \frac 1p - \frac{1}{2}$ and $t>\frac 12$. Then we can estimate
\begin{equation*}
	\E e_{\delta,\zeta}(K^*N) \lesssim \delta^{1-1/r} \zeta^\frac{1-q}{r},
\end{equation*}
where $r= q(t+s-\frac 12)+1>\frac q2$.
\end{lemma}

\begin{proof}
By definition we have
\begin{eqnarray*}
e_{\delta,\zeta}(K^* N) & = &
\inf_{w\in Y} \left(\zeta R^\star\left(\frac{K^*(w-N)}{\zeta}\right) + \frac \delta 2 \norm{w}^2_Y\right) \\
& = & \frac {1}{2 q}  \inf_{w\in Y} \sum_{\ell=1}^\infty 
\left(2 \zeta^{1-q} \ell^{q(-s+\frac 12)-1} k_\ell^q |w_\ell - N_\ell|^q
+  \delta w_\ell^2 \right). \\
\end{eqnarray*}
Let us now abbreviate $a_\ell = 2 \zeta^{1-q}\ell^{q(-s+\frac 12)-1} k_\ell^q$ and consider 
an upper bound for the infimum by elements in $Y$ that are supported only on the first $L$ basis vectors. We find that
\begin{equation*}
	e_{\delta,\zeta}(K^* N)
	 \lesssim \sum_{\ell=1}^L \min\{a_\ell |N_\ell|^q, \delta N_\ell^2\}  +  \sum_{\ell=L+1}^\infty a_\ell |N_\ell|^q
\end{equation*}
and, therefore,
\begin{equation}
	\label{eq:expec_eterm}
	\E e_{\delta,\zeta}(K^* N) \lesssim  \sum_{\ell=1}^L \E \min\{a_\ell |N_\ell|^q, \delta N_\ell^2\}  +  \sum_{\ell=L+1}^\infty a_\ell
\end{equation}
since $\E |N_\ell|^q \simeq 1$.
In order to evaluate the expectation in \eqref{eq:expec_eterm} we need the following integral identity
\begin{equation}
	\int_D^\infty x^2 \exp\left(-\frac{x^2}{2}\right) dx = D \exp\left(-\frac{D^2}{2}\right) + \sqrt{\frac{\pi}{2}} {\rm erfc} \left(\frac{D}{\sqrt 2}\right)
\end{equation}
and the estimate
\begin{align}
\int_0^D x^q \exp\left(-\frac{x^2}{2}\right) dx = g(D)  \lesssim
  \begin{cases}
    D^{q+1}       & \quad \text{when}\,0\leq D\leq1\\
   1  & \quad D>1.
  \end{cases}
\end{align}
Next we define $D_\ell$ to satisfy
\begin{equation*}
 	a_\ell D_\ell^q = \delta D_\ell^2, \quad {\rm i.e.}\quad D_\ell = \left(\frac{\delta}{a_\ell}\right)^\frac{1}{q-2}.
\end{equation*}
The expectation in \eqref{eq:expec_eterm} satisfies
\begin{eqnarray*}
\E \min\{a_\ell |N_\ell|^q, \delta N_\ell^2\}  &	= & \int_{-\infty}^\infty \min \{a_\ell |x|^q, \delta x^2\} \exp\left(-\frac{x^2}{2}\right) dx  \\
	& = & 2a_\ell\int_{0}^{D_\ell} |x|^q \exp\left(-\frac{x^2}{2}\right) dx + 2 \delta \int_{D_\ell}^\infty x^2 \exp\left(-\frac{x^2}{2}\right) dx \\
	& \lesssim & a_\ell g(D_\ell)  + \delta(D_\ell+1) f(D_\ell),
\end{eqnarray*}
where 
\begin{align*}
f(D_\ell) \simeq
  \begin{cases}
    1,       & \quad \text{when}\, 0\leq D_\ell\leq 1 \, \text{and}\\
   \exp(-\frac{D_\ell^2}{2})  & \quad \text{otherwise.}
  \end{cases}
\end{align*}
For small values of $D_\ell$, i.e. $\delta \leq a_\ell$ we have
\begin{equation*}
	\E \min\{a_\ell |N_\ell|^q, \delta N_\ell^2\}  \lesssim \frac{\delta^{q+1}}{a_\ell^q} + \left(\frac{\delta}{a_\ell}\right)^\frac{1}{q-2} + \delta
\end{equation*}
 and for $\delta>a_\ell$ it holds that
 \begin{equation*}
 	\E \min\{a_\ell |N_\ell|^q, \delta N_\ell^2\}  \lesssim a_\ell + \delta\left(\left(\frac{\delta}{a_\ell}\right)^\frac{1}{q-2}+1\right) \exp\left(\left(\frac{\delta}{a_\ell}\right)^\frac{1}{q-2}\right).
 \end{equation*}
 Since $L$ in \eqref{eq:expec_eterm} is arbitrary, we have
 \begin{eqnarray}
 	\label{eq:e_est_by_three_terms}
 	\E e_{\delta,\zeta}(K^* N) & \lesssim  & \sum_{\ell=1}^\infty \E \min\{a_\ell |N_\ell|^q, \delta N_\ell^2\}  \nonumber \\
 	& =  & \sum_{\ell=1}^{\widetilde L} \left(\frac{\delta^{q+1}}{a_\ell^q} + \left(\frac{\delta}{a_\ell}\right)^\frac{1}{q-2} + \delta\right) + 
 	\sum_{\ell=\widetilde L}^\infty \left(a_\ell + \delta \left(\left(\frac{\delta}{a_\ell}\right)^\frac{1}{q-2}+1\right) \exp\left(\left(\frac{\delta}{a_\ell}\right)^\frac{1}{q-2}\right)\right)\nonumber \\
 	& \lesssim & \delta \widetilde L   + \sum_{\ell=\widetilde L}^\infty a_\ell + \delta \sum_{\ell=\widetilde L}^\infty  \left(\left(\frac{\delta}{a_\ell}\right)^\frac{1}{q-2}+1\right) \exp\left(\left(\frac{\delta}{a_\ell}\right)^\frac{1}{q-2}\right),
 \end{eqnarray}
where we have chosen $\widetilde L$ so that $a_{\widetilde L+1} < \delta \leq a_{\widetilde L}$. 

Recall now that due to our assumptions we have
\begin{equation*}
	a_\ell \simeq \zeta^{1-q} \ell^{q(-t-s+\frac 12)-1} = \zeta^{1-q} \ell^{-r},
\end{equation*}
where we write $r = q(t+s-\frac 12)+1>\frac q2$.
Since $\delta \simeq a_{\widetilde L}$, our choice for $\widetilde L$ indicates that
\begin{equation*}
	\widetilde L \simeq \delta^{-\frac{1}{r}} \zeta^{-\frac{q-1}{r}}.
\end{equation*}
Now we are able to estimate all terms in \eqref{eq:e_est_by_three_terms}. First, we have
\begin{equation*}
	\sum_{\ell=\widetilde L}^\infty a_\ell \simeq 
	\zeta^{1-q} \int_{\widetilde L}^\infty \ell^{-r}d\ell
	= \zeta^{1-q} \frac{(\delta^{-\frac{1}{r}} \zeta^{\frac{1-q}{r}})^{1-r}}{r-1} \simeq \delta \widetilde L.
\end{equation*}
Second, by denoting  $\theta = \left(\frac{\delta}{\zeta^{1-q}}\right)^\frac{1}{q-2} $ we obtain
\begin{eqnarray*}
	\sum_{\ell=\widetilde L}^\infty \left( \frac \delta{a_\ell}\right)^\frac{1}{q-2} f\left(\left( \frac \delta{a_\ell}\right)^\frac{1}{q-2}\right) 
	& \simeq & \theta \int_{\widetilde L}^\infty x^\frac{r}{q-2} \exp\left(-\frac{\theta^2}{2} x^\frac{2r}{q-2}\right) dx \\
	& = & \theta^\frac{2-q}{r} \int_1^\infty y^\frac{r}{q-2} \exp\left(\frac{-y^\frac{2r}{q-2}}{2}\right) dy \simeq \delta^{-1/r} \zeta^\frac{1-q}{r} \simeq \widetilde L,
\end{eqnarray*}
where we applied a change of variables $y= \theta^\frac{q-2}{r} x$. 
Third, we notice similarly to the second case that 
\begin{equation*}
	\sum_{\ell=\widetilde L}^\infty  f\left(\left( \frac \delta{a_\ell}\right)^\frac{1}{q-2} \right)
	 \simeq   \int_{\widetilde L}^\infty  \exp\left(-\frac{\theta^2}{2} x^\frac{2r}{q-2}\right) dx \\
	\simeq  \theta^\frac{2-q}{r} \int_1^\infty  \exp\left(-\frac{y^\frac{2r}{q-2}}{2}\right) dy\simeq \widetilde L.
\end{equation*}
Finally, by applying the three estimates above to \eqref{eq:e_est_by_three_terms} we conclude that
\begin{equation*}
		\E e_{\delta,\zeta}(K^* N)  \lesssim \delta\widetilde L		\simeq \delta^{1-1/r} \zeta^\frac{1-q}{r},
\end{equation*}
where $r= q(t+s-1/2)+1$, which yields the claim.
\end{proof}

\begin{thm} 
\label{thm:wavelet_casep}
Let us assume that $K$ and $R$ are given as above, $s\geq \frac 1p - \frac{1}{2}$ and $t>\frac 12$. For the choice 
$\alpha \simeq \delta^\kappa$, where
\begin{equation*}
	\kappa =  \frac{4(s+t)}{4(s+t)+1}
\end{equation*}
we obtain the convergence rate
\begin{equation*}
	\E D_R^{\mu_\alpha^\delta,\mu^\dagger}(U_\alpha^\delta,u^\dagger) \lesssim \delta^{\kappa}.
\end{equation*}
\end{thm}

\begin{proof}
By combing Lemma \ref{lem:besov_p_e_estimate} with inequality \eqref{eq:p>1_first_breg_estimate} we have 
\begin{eqnarray}
\E D_R^{\mu_\alpha^\delta,\mu^\dagger}(u_\alpha^\delta,u^\dagger) &\lesssim &
	\alpha+ \inf_{\zeta>0} \left(
	\zeta^{\frac 2{2-p}} \left(\frac \delta\alpha\right)^{\frac 2{2-p}} 
	\left(1 +\frac \delta \alpha (\frac \delta 2)^{1-\frac 1r}\left(\frac{\alpha\gamma}{\delta}\right)^\frac{1-q}{r} \right)
	+ \frac \delta\alpha \delta^{1-1/r} \zeta^\frac{1-q}{r}\right) \nonumber \\
	& =: & \widetilde M_1+\widetilde M_2. \label{eq:thm_besov_p>1_first_est}
\end{eqnarray}
By setting $\alpha\simeq\delta^\kappa$ we can write 
\begin{align*}
\widetilde M_2 & \lesssim  \inf_{\zeta>0} \left(
	\zeta^\nu \delta^{\nu(1-\kappa)}
	\left(1 + \delta^{2+\frac{q-2}{r}-\kappa(1+\frac{q-1}{r})}\right)+\delta^{2-\kappa-\frac{1}{r}}\zeta^\frac{1-q}{r} \right)\\
	& = \left(1 + \delta^{2+\frac{q-2}{r}-\kappa(1+\frac{q-1}{r})}\right)^{\frac{q-1}{\nu r+q-1}}\delta^{\frac{\nu(1-\kappa)(q-1)+\nu r(2-\kappa-\frac{1}{r})}{\nu r+q-1}},
\end{align*}
where we have denoted $\nu=\frac{2}{2-p}$.
By our assumption $r>1$, that is, $2+\frac{q-2}{r}-\kappa(1+\frac{q-1}{r})>0$ and we obtain
\begin{align*}
M_2\lesssim \delta^{\frac{\nu(1-\kappa)(q-1)+\nu r(2-\kappa-\frac{1}{r})}{\nu r+q-1}}. 
\end{align*}
By optimizing the convergence rate in \eqref{eq:thm_besov_p>1_first_est} we have
\begin{align*}
\E D_R^{\mu_\alpha^\delta,\mu^\dagger}(u_\alpha^\delta,u^\dagger)  &\lesssim
\delta^{\frac{\nu(q+2(r-1))}{\nu(q+2r-1)+q-1}}. 
\end{align*}
Since $r= q(t+s-1/2)+1$ we can write 
\begin{align*}
\frac{\nu(q+2(r-1))}{\nu(q+2r-1)+q-1}=\frac{\nu(2q(t+s))}{\nu(2q(t+s)+1)+q-1}. 
\end{align*}
Furthermore since $\nu=\frac{2}{2-p}$ and $\frac{1}{p}+\frac{1}{q}=1$ we see that 
\begin{align*}
\frac{4q}{2-p}=\frac{4p}{(2-p)(p-1)}\quad \text{and}\quad \frac{2}{2-p}+q-1=\frac{p}{(2-p)(p-1)},
\end{align*}
 which yields the claim
\begin{align*}
\E D_R^{\mu_\alpha^\delta,\mu^\dagger}(u_\alpha^\delta,u^\dagger) \lesssim
\delta^{\frac{4(s+t)}{4(s+t)+1}}. 
\end{align*}
\end{proof}

\begin{rem}
If we assume $p=2$ then we can use the inequalities of Theorem 3.10 instead of Theorem 3.7 to attain  the same result. Note that if $p=2$ and the  exact source condition is assumed then $\mu^\dagger=(I-\Delta)^su^\dagger=K^*w$, where $w\in L^2(\T)$. This means that $u^\dagger\in H^{r}(\T)$, with $r=2s+t$, and we can write
\begin{align*}
\E D_R^{\mu_\alpha^\delta,\mu^\dagger}(u_\alpha^\delta,u^\dagger)  &=
\E \|u_\alpha^\delta-u^\dagger\|_{H^s(\mathbb{T})}^2\leq
\delta^{\frac{4(r-s)}{2r+2t+1}},
\end{align*}
which is the minimax rate \cite{Cavalier2008}. 
\end{rem}

\begin{rem}
We point out that minimax rates for linear statistical inverse problems in wavelet basis have been studied for estimators based on Galerkin methods and non-linear thresholding algorithms (see \cite{Donoho1995,cohen2004adaptive,hohage2016inverse} and references therein). In the first two papers the authors construct a finite-dimensional estimator $u_\delta$ for any $\delta>0$ such that
\begin{equation}
	\label{eq:galerkin_rate}
	\sup_{u\in B} \E \norm{u-u_\delta}^2_{L^2} \lesssim \left(\delta \sqrt{|\log \delta|}\right)^{\frac{4s}{2s+2t+1}},
\end{equation}
the forward operator $K$ is $t$ times smoothing (similar to \eqref{eq:besov_smooth_cond}) and 
\begin{equation}
	\label{eq:galerkin_methods_source}
	B = \{u \; | \; \norm{u}_{B^s_p} \leq C\}.
\end{equation}
Such rates are also known to be optimal \cite{cohen2004adaptive}.
Compared to \eqref{eq:galerkin_methods_source} our method builds upon a more general source condition. We do not necessarily require that the true solution is in the range of $K^*$ (if the range is defined as $K^* Y$ and not $K^*$ on a larger space including the noise). However, there is interplay between smoothness of $K$ and our source condition. In addition, the rate in \eqref{eq:galerkin_rate} is achieved in a $L^2$-norm, whereas the symmetric Bregman distance of $B^s_1$-norm in Theorem \ref{thm:wavelet} is not a norm, since it is not strictly positive and does not satisfy a triangle inequality. On the other hand the Bregman distance estimate can be used to provide structural properties related to sparsity, e.g. a bound on the norm of the wavelet coefficients of the reconstruction outside the support of the coefficients of $u^\dagger$ (cf. \cite{bur07}). We also mention that for the related approach wavelet soft-thresholding, where first a reconstruction $K^{-1}f^\delta$ is computed in a very large Besov space and then projected back by soft thresholding of the wavelet coefficients (cf. \cite{Donoho1995}), our approach can be used to provide analogous rates as in \cite{cohen2004adaptive} by applying the estimates to the variational regularization
$$ J_\alpha^\delta(u) = \frac{1}2 \Vert u \Vert_{L^2} - \langle (K^{*})^{-1} u, f^\delta \rangle + \alpha R(u), $$
where $R(u)$ is the associated Besov norm (weighted $\ell^1$-norm on wavelet coefficients). This is indeed a special case of our approach with definition $\Vert f \Vert_Y = \Vert K^{-1} f \Vert_{L^2}$, that is, $Y$ is the  space of elements where the latter norm is finite. Note that this corresponds naturally to a large noise case that cannot be treated with the existing theory. We also mention that some extensions to the case of $K$ not being injective are possible.
From our analysis and the form of the functional it becomes apparent that a necessary condition is that the true solution is in the range of $K^*$ and the associated subgradient is in $L^2$, which is indeed a rather weak condition. Our approach then yields $L^2$-estimates as in \cite{cohen2004adaptive} (corresponding to an estimate for $\Vert K(u-u^\dagger)\Vert^2_Y$, but in addition we also obtain an estimate in the Bregman distance providing information about the sparsity. 
It remains an interesting future question to provide more comparison between the Galerkin approach and variational methods.
\end{rem}

\subsection{Total Variation-type Regularization} 

In the following we discuss the case of total variation regularization 
\begin{equation}
\label{eq:TV_norm}
R(u) = \sup_{\varphi \in C_0^\infty(\T^d), \Vert \varphi \Vert_\infty \leq 1} \int_{\T^d}
\nabla \cdot \varphi u~dx,
\end{equation}
or related regularizations such as infimal convolutions with higher order total variation (cf. \cite{tvzoo} and references therein) in the case of spatial dimension $d\leq 2$, when there is an embedding into $X_0=L^2(\T^d)$. Thus, it is natural to use Hilbert space embedding in this case. We assume that $K$ can be extended to a $t>d/2+\epsilon$ times smoothing bijective bounded linear operator in Sobolev scale. We also assume that $N$ is white noise taking values in $H^{-d/2-\epsilon}$ as in the previous sections.  We will use the estimate \eqref{eq:embedding_estimate} for realizations of $K^* N$ (noting $L=K^*K$) and write
$$  
	e_{\delta,\zeta}(K^*N) \leq C \Vert L^{-\nu}K^* N  \Vert_{L^2(\T^d)}^{\frac{1}\nu}  \zeta^{-\frac{1-2\nu}{\nu}}\delta
$$
to obtain an estimate for the expectation $\E(e_{\delta,\zeta}(K^*N))$. Subsequently one could use similar reasoning as in the previous section respectively Section 3.1 to obtain full rates, which we leave to the reader. 
 
The key question for the finite expectation of $e_{\delta,\zeta}(K^*N)$ is the choice of $\nu$ such that
$$\E \Vert L^{-\nu}K^* N  \Vert_{L^2(\T^d)}^{\frac{1}\nu}  < \infty. $$
Note that by Fernique's theorem any moment of white noise is finite in $H^{-d/2-\epsilon}(\T^d)$ for any $\epsilon>0$ \cite{daprato}. Thus, we do not need to worry about the exponent $\frac{1}\nu$ in the expectation, but rather optimize $\nu$ to  have
$$  \Vert L^{-\nu}K^* N  \Vert_{L^2(\T^d)} \leq C \Vert N \Vert_{H^{-d/2-\epsilon}(\T^d)}. $$
With the above smoothing assumptions, we see that $K^*$ maps from $H^{-d/2-\epsilon}(\T^d)$ to $H^{t-d/2-\epsilon}(\T^d)$, hence we achieve the above rate estimate if $L^{-\nu}$ is bounded from $H^{t-d/2-\epsilon}(\T^d)$ to $L^2(\T^d)$. For $K$ being the inverse of a translation invariant differential or pseudo-differential operator one obtains that $L^\nu:L^2(\T^d)\rightarrow H^{2t\nu}(\T^d)$. 
In the following we write $K \in \Psi^{\rho}$ for a pseudodifferential operator $K$ if its symbol is in ${\mathcal S}^{\rho}(\T^d; \T^d)$ \cite{taylor1996pseudodifferential}. The condition above means that we should choose $\nu = \frac{t-d/2-\epsilon}{2t}$. 

As a specific example consider the pseudodifferential operator $K=(-\Delta + I)^{-1}$. Then $K$ is a  twice smoothing bijective operator between  $H^{-d/2-\epsilon}(\T^d)$ and $H^{2-d/2-\epsilon}(\T^d)$, which gives us $\nu = \frac{2-d/2-\epsilon}4$, i.e., one can choose $\nu$ arbitrarily close to $\frac{4-d}8$.

\begin{thm}
Let us assume that $K\in\Psi^{-t}$ where $t>d/2+\epsilon$ with some $\epsilon>0$, that is, $K$ is of order $t$ smoothing pseudodifferential operator. Regularization functional  $R$ is defined by \eqref{eq:TV_norm} and $\mu^\dagger$ satisfies the approximate source condition of order $r_1\geq 0$ in Assumption \ref{assumption:p=1}. Then for the choice $\alpha \simeq \delta^\kappa$ where
\begin{equation*}
\kappa =
\begin{cases}
\frac{1+r_1}{(2+r_1)(1-\nu)} & {\rm for}\; r_1 \leq \frac{d+2\epsilon}{t}\; {\rm and}\\
1 & {\rm else}\;
\end{cases}
\end{equation*}
we obtain the convergence rate
\begin{equation*}
	\E D_R^{\mu_\alpha^\delta,\mu^\dagger}(U_\alpha^\delta,u^\dagger) \lesssim 
\begin{cases}
\delta^{\frac{1}{(2+r_1)(1-\nu)}}\leq \delta^\frac{1}{2+r_1} & {\rm for}\; r_1 \leq \frac{d+2\epsilon}{t}\; {\rm and}\\
\delta^\frac{1}{1+r_1} & {\rm else}\;
\end{cases}
\end{equation*}
where $\nu=\frac{t-d/2-\epsilon}{2t}$.
\end{thm}

\begin{proof}
Recall that
\begin{equation}
	\E D_R^{\mu_\alpha^\delta,\mu^\dagger}(U_\alpha^\delta,u^\dagger)
	\leq \widetilde M_1 + \widetilde M_2.
\end{equation}
where terms $\widetilde M_1$ and $\widetilde M_2$ are given in equation \eqref{eq:besov_proof_1homog}.
We have for a constant $\gamma $ and $\alpha \simeq \delta^\kappa$, $\kappa\leq 1$ that
\begin{equation*}
	\E e_{\delta, \frac{\alpha\gamma}{\delta}}(K^*N) \lesssim \gamma^{2-\frac 1 \nu} \delta^{\kappa + (\frac 1 \nu  -1)(1-\kappa)}
	\lesssim 1.
\end{equation*}
since $\frac 1 \nu > 2$.
Therefore, we obtain
\begin{equation*}
	\widetilde M_1 \lesssim \inf_{\zeta_1\in\R_+} \left\{\zeta_1 + \delta^\kappa \zeta_1^{-r_1}\right\}
	\simeq \delta^{\frac{\kappa}{1+r_1}}
\end{equation*}
and
\begin{equation*}
	\widetilde M_2 \lesssim \inf_{\zeta_2\in\R_+} \left\{\delta^{1-\kappa}\left(\zeta_2 + \delta \zeta_2^{-\frac{1-2\nu}{\nu}}\right) \right\} \simeq \delta^\frac{1-\kappa(1-\nu)}{1-\nu}.
\end{equation*}
If $r_1 \leq \frac{d+2\epsilon}{t}$ the convergence rate is minimized with $\kappa$ which satisfies
\begin{equation*}
	\kappa = \frac{1+r_1}{(2+r_1)(1-\nu)}.
\end{equation*}
For $r_1 \geq \frac{d+2\epsilon}{t}$ we choose $\kappa=1$ and consequently the claim holds.
\end{proof}

\section{Outlook to the Bayesian approach}
\label{sec:outlook}

In the Bayesian approach to inverse problems the model equation \eqref{eq:main_noisy} is often written in the form
\begin{equation}
\label{eq:bayes}
F_\delta=KU+\delta N
\end{equation}
where, in addition to the observational random noise $N$, we describe our prior beliefs about the unknown in terms of the probability distribution of the random variable $U:\Omega \to X$. 
The solution to the inverse problem is then the probability distribution of $U$ conditioned on a measurement outcome $F_\delta$.
The posterior distribution now provides means for uncertainty quantification. 

The analysis of small noise limit, in Bayesian case also known as the theory of posterior consistency, has attracted a lot of interest in the last decade. Posterior convergence rates were first studied in \cite{Ghosal2000,Shen2001}. In those two papers Gaussian noise and prior are assumed and the interest is on the convergence of the approximated solution $U_\delta^\alpha$, generated by a 'true' $u^\dagger$, to the same truth $u^\dagger$. Similar convergence or the contraction of the whole posterior distribution is further studied e.g. in papers \cite{Agapiou2013, Dashti13, Huang2004, Monard2017, Ray2013, Vollmer2013}. In \cite{kekkonen14,kekkonen15} Bayesian cost estimator similar to (\ref{eq:Bayesian_cost}) in Gaussian case is considered. 


A widely used approach to extract information from a posteriori distribution is to find so-called maximum a posteriori (MAP) estimator. In finite dimensional problems, the MAP estimate maximizes a posteriori probability density function and is, loosely speaking, the most probable solution to the problem \eqref{eq:bayes}. In the infinite-dimensional case, the MAP estimator is less understood. In certain probabilistic models, the MAP estimate is known to minimize a problem of type \eqref{minimisation}. We refer to our earlier work in \cite{HelinBurger15,HL11} and other authors in \cite{Dashti13, DS15} for more discussion on the topic. 
We point out that, in general, the connection between the estimator induced by \eqref{minimisation} and the MAP estimate is not well-established.
Despite this deficit, understanding the Bayes cost in such a case based on Bregman distance would be highly interesting for practical problems.

Our results in Theorem \ref{deterministicestimates} now directly yields that  
\begin{eqnarray}\label{eq:Bayesian_cost}
\E_{N,U}(	D_R^{\mu_\alpha^\delta,\mu}(U_\alpha^\delta,U) ) &\leq& 
 \E_N\Big(\E_{U}\Big(\inf_{(\zeta_1,\zeta_2) \in (\mathbb{R}^+)^2} (   \zeta_1 + \frac{\delta}\alpha \zeta_2 )^{1/(1-\theta)} C_\theta(U_\alpha^\delta ,U)^{1/(1-\theta)} + \nonumber \\ &&\qquad \qquad \frac{1}{1-\theta}
 e_{\alpha,\zeta_1}(M)+ \frac{\delta}{\alpha(1-\theta)} e_{\delta,\zeta_2}(K^*N)\Big)\Big),
\end{eqnarray}
where $M:\Omega \to X^*$ formally satisfies $M(\omega) \in \partial R(U(\omega))$.
The Bayes cost for the MAP estimate, however, is not a straightforward matter since the subgradient set $\partial R(U)$ is not necessarily well-defined. Consider a Gaussian prior $U$ in a Hilbert space $X$ with zero-mean and covariance $C_U : X \to X$. In such a case, the functional $R$ induced by the prior satisfies $R(u) = \norm{C_U^{-1/2} u}^2_X$, i.e., $R$ coincides with the norm of the Cameron--Martin space. On the other hand, realizations of $U$ are in the Cameron--Martin space with probability zero.
Similarly, expectation over $R$ and Bregman distance in \eqref{eq:Bayesian_cost} are not bounded.

It is know from the earlier work \cite{kekkonen15} by the last author that
in Gaussian setting the Bregman distance based Bayes cost can be estimated using a weaker norm than the one induced by the prior. Hence an intriguing question for future work is to characterize functional $R$ for which the Bayes cost (and the bound) in \eqref{eq:Bayesian_cost} makes sense.

Let us finally comment that in a purely Bayesian approach the prior information should be independent of the measurement $F_\delta$. 
For instance, MAP estimate of problem \eqref{eq:main_noisy} for a $\delta$-independent prior and a noise distribution $\delta N$ with white noise $N$ formally correspond to an estimator \eqref{minimisation} where $\alpha$ is replaced by $\alpha \delta^2$ for a constant $\alpha$. In literature this principle is occasionally omitted and general a priori rules $\alpha=\alpha(\delta)$ are considered.  Such an approach resembling the frequentist method leads to 'priors' that are scaled with respect to the noise level $\delta$ and hence no longer independent of the measurement. With general $\alpha(\delta)$ the minimisation problem \eqref{minimisation} can not be seen as a proper MAP estimate. However, it is a useful estimator to study since with constant $\alpha$ we often do not get convergence in the original space.

\section*{Acknowledgements}  
This work has been supported by the German Science Exchange Foundation DAAD via Project
57162894, Bayesian Inverse Problems in Banach Space.
MB acknowledges further support by  ERC via Grant EU FP 7 - ERC Consolidator Grant 615216 LifeInverse. 
TH and HK were supported by Academy of Finland via grants 275177 and 285463 and Finnish Centre of Excellence in Inverse Problems Research 2012-2017 (CoE-project 284715), respectively.
HK was further supported by Emil Aaltonen Foundation and EPSRC via Grant EP/K034154/1. The authors thank Peter Mathe (WIAS Berlin) for useful links to literature.

\appendix

\section{Convex Conjugates}\label{sec:convexconjugates}

For completeness we recall  the convex conjugate $R^\star: X^* \rightarrow \R \cup \{\infty \} $ defined via
\begin{equation}
	R^\star(q) = \sup_{u \in X} \left( \langle q, u \rangle_{X^* \times X} - R(u) \right).
\end{equation}
Note that by definition of $R^\star(q)$ we obtain the following well-known generalization of Young's inequality 
\begin{equation}
	 \langle q, u \rangle_{X^* \times X} \leq R(u) + R^\star(q),
\end{equation}
for all $u \in X$ and $q\in X^*$, which we employ at several instances throughout the paper.

\bibliographystyle{abbrv}

\bibliography{references}

\end{document}